\documentclass[english]{article}
\usepackage[T1]{fontenc}
\usepackage[latin9]{inputenc}
\usepackage{geometry}
\usepackage{amsmath}
\usepackage{amsthm}
\usepackage{amssymb}
\usepackage{graphicx}
\usepackage{wasysym}

\makeatletter

\providecommand{\tabularnewline}{\\}

\numberwithin{equation}{section}
\numberwithin{figure}{section}
  \theoremstyle{plain}
  \newtheorem*{conjecture*}{\protect\conjecturename}
\theoremstyle{plain}
\newtheorem{thm}{\protect\theoremname}
  \theoremstyle{plain}
  \newtheorem{lem}[thm]{\protect\lemmaname}
  \theoremstyle{plain}
  \newtheorem{conjecture}[thm]{\protect\conjecturename}
  \theoremstyle{definition}
  \newtheorem{defn}[thm]{\protect\definitionname}
  \theoremstyle{remark}
  \newtheorem*{acknowledgement*}{\protect\acknowledgementname}

\usepackage{hyperref}
\hypersetup{
    colorlinks,
    allcolors=red
}
\usepackage[lined,boxed,ruled,norelsize,algo2e]{algorithm2e}
\@addtoreset{section}{part}
\usepackage{tikz}
\usepackage{wrapfig}

\usepackage{subfig}

\makeatother

\usepackage{babel}
  \providecommand{\acknowledgementname}{Acknowledgement}
  \providecommand{\conjecturename}{Conjecture}
  \providecommand{\definitionname}{Definition}
  \providecommand{\lemmaname}{Lemma}
\providecommand{\theoremname}{Theorem}

\begin{document}
\global\long\def\defeq{\stackrel{\mathrm{{\scriptscriptstyle def}}}{=}}
\global\long\def\norm#1{\left\Vert #1\right\Vert }
\global\long\def\R{\mathbb{R}}
\global\long\def\Ent{\mathrm{Ent}}
\global\long\def\vol{\mathrm{vol}}
 \global\long\def\Rn{\mathbb{R}^{n}}
\global\long\def\tr{\mathrm{Tr}}
\global\long\def\diag{\mathrm{diag}}
\global\long\def\cov{\mathrm{Cov}}
\global\long\def\E{\mathbb{E}}
\global\long\def\P{\mathbb{P}}
\global\long\def\Var{\mbox{Var}}
\global\long\def\rank{\text{rank}}
\global\long\def\lref#1{\text{Lem }\ltexref{#1}}
\global\long\def\lreff#1#2{\text{Lem }\ltexref{#1}.\ltexref{#1#2}}
\global\long\def\ltexref#1{\ref{lem:#1}}\global\long\def\ttag#1{\tag{#1}}
\global\long\def\cirt#1{\raisebox{.5pt}{\textcircled{\raisebox{-.9pt}{#1}}}}
\global\long\def\op{\mathrm{op}}

\definecolor{c063cb8}{RGB}{0,64,192}
\definecolor{cff0000}{RGB}{255,0,0}
\definecolor{c009200}{RGB}{0,128,0}
\definecolor{c008000}{RGB}{0,128,0}
\definecolor{cff6666}{RGB}{255,128,128}
\definecolor{c008600}{RGB}{0,128,0} 

\title{The Kannan-Lovász-Simonovits Conjecture\thanks{This is a \textbf{working draft}. We will be grateful for any comments!}}

\author{Yin Tat Lee\thanks{University of Washington, yintat@uw.edu}, Santosh
S. Vempala\thanks{Georgia Tech, vempala@gatech.edu}}
\maketitle
\begin{abstract}
The Kannan-Lovász-Simonovits conjecture says that the Cheeger constant
of any logconcave density is achieved to within a universal, dimension-independent
constant factor by a hyperplane-induced subset. Here we survey the
origin and consequences of the conjecture (in geometry, probability,
information theory and algorithms) as well as recent progress resulting
in the current best bounds. The conjecture has lead to several techniques
of general interest.
\end{abstract}

{\small{}\tableofcontents{}}{\small \par}

\pagebreak{}

\section{Introduction}

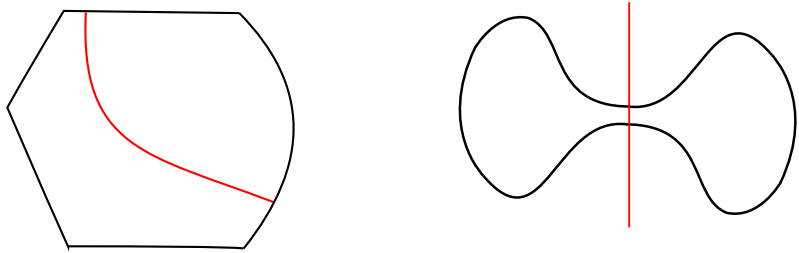
\begin{figure}
\centering
\subfloat{
\begin{tikzpicture}[y=0.80pt, x=0.80pt, yscale=-0.500000, xscale=0.500000, inner sep=0pt, outer sep=0pt]   \path[draw=black,line join=miter,line cap=butt,even odd rule,line width=0.800pt]     (413.0778,258.6109) .. controls (413.0778,258.6109) and (413.0778,258.6109) ..     (247.0446,256.5289) .. controls (193.7500,346.0581) and (193.7500,348.1402) ..     (193.7500,348.1402) .. controls (251.1442,481.3930) and (251.1442,479.3109) ..     (251.1442,479.3109) .. controls (417.1774,479.3109) and (417.1774,481.3930) ..     (417.1774,481.3930);   \path[draw=black,line join=miter,line cap=butt,even odd rule,line width=0.800pt]     (413.0778,258.6109) .. controls (481.0505,327.4657) and (480.5875,402.0638) ..     (417.1774,481.3930);   \path[draw=cff0000,line join=miter,line cap=butt,even odd rule,line     width=0.800pt] (267.9571,257.8746) .. controls (261.7490,388.1156) and     (325.7714,392.2116) .. (445.0982,437.3601);
\end{tikzpicture}
}\quad\quad\quad\quad
\subfloat{
\begin{tikzpicture}[y=0.80pt, x=0.80pt, yscale=-0.500000, xscale=0.500000, inner sep=0pt, outer sep=0pt]   \begin{scope}[cm={{0.5852,0.0,0.0,1.21096,(51.29251,-57.70757)}}]     \path[line cap=butt,miter limit=4.00,nonzero rule,line width=0.800pt]       (136.3636,311.4531) ellipse (1.0904cm and 0.9621cm);     \path[draw=black,line join=miter,line cap=butt,even odd rule,line width=0.950pt]       (332.3096,374.9221) .. controls (200.0131,374.4592) and (231.5385,317.4066) ..       (170.1435,305.7865) .. controls (117.7513,300.2597) and (83.3933,329.3002) ..       (83.3933,329.3002) -- (83.3930,329.3002) .. controls (21.9907,392.6421) and       (97.1970,429.0782) .. (97.1970,429.0782);     \path[draw=black,line join=miter,line cap=butt,even odd rule,line width=0.950pt]       (329.3556,388.8739) .. controls (213.2139,381.7145) and (196.0451,485.7344) ..       (97.1970,429.0782);   \end{scope}   \begin{scope}[cm={{-0.5852,0.0,0.0,-1.21096,(437.87719,867.20595)}}]     \path[line cap=butt,miter limit=4.00,nonzero rule,line width=0.800pt]       (136.3636,311.4531) ellipse (1.0904cm and 0.9621cm);     \path[draw=black,line join=miter,line cap=butt,even odd rule,line width=0.950pt]       (332.3096,374.9221) .. controls (200.0131,374.4592) and (231.5385,317.4066) ..       (170.1435,305.7865) .. controls (117.7513,300.2597) and (83.3933,329.3002) ..       (83.3933,329.3002) -- (83.3930,329.3002) .. controls (21.9907,392.6421) and       (97.1970,429.0782) .. (97.1970,429.0782);     \path[draw=black,line join=miter,line cap=butt,even odd rule,line width=0.950pt]       (329.3556,388.8739) .. controls (213.2139,381.7145) and (196.0451,485.7344) ..       (97.1970,429.0782);   \end{scope}   \path[draw=cff0000,line join=miter,line cap=butt,even odd rule,line     width=0.688pt] (246.0550,297.5922) .. controls (246.0550,510.5716) and     (246.0550,510.5716) .. (246.0550,510.5716) -- (246.0550,510.5716);
\end{tikzpicture}
}\caption{Good and bad isoperimetry\label{fig:Good-and-bad}}
\end{figure} 

In this article, we describe the origin and consequences of the Kannan-Lovász-Simonovits
(KLS) conjecture, which now plays a central role in convex geometry,
unifying or implying older conjectures. Progress on the conjecture
has lead to new proof techniques and influenced diverse fields including
asymptotic convex geometry, functional analysis, probability, information
theory, optimization and the theory of algorithms. 

\subsection{The KLS conjecture}

The isoperimetric problem asks for the unit volume set with minimum
surface area. For Euclidean space, ancient Greeks (around 150 BC \cite{ashbaugh2010problem})
knew that the solution is a ball; a proof was only found in 1838 by
Jakob Steiner \cite{steiner1838einfache}. For sets of arbitrary volume,
the isoperimetry (or expansion) of the set is defined to be the ratio
of surface area to its volume (or its complement, whichever is smaller).
For the Gaussian distribution, or the uniform distribution over a
Euclidean sphere, the minimum isoperimetric ratio is achieved by a
halfspace, i.e., a hyperplane cut \cite{sudakov1974extremal,borell1975brunn}.
However, this is not true in general (even for the uniform distribution
over a simplex) and the minimum ratio set, in the worst case, can
be very far from a hyperplane. In general, any domain that can be
viewed as two large parts with a small boundary between them, a ``dumbbell''-like
shape, can have arbitrary isoperimetric ratio (Figure \ref{fig:Good-and-bad}).
It is natural to expect that convex bodies and logconcave functions
(whose logarithms are concave along every line) have good isoperimetry
\textemdash{} they cannot look like dumbbells. The KLS conjecture
says that a hyperplane cut achieves the minimum ratio up to a constant
factor for the uniform distribution on any convex set, and more generally
for any distribution on $\R^{n}$ with a logconcave density. The constant
is universal and independent of the dimension.

Formally, for a density $p$ in $\R^{n}$, the measure of a set $S\subseteq\R^{n}$
is $p(S)=\int_{S}p(x)\,dx.$ The boundary measure of this subset is
\[
p(\partial S)=\inf_{\varepsilon\rightarrow0^{+}}\frac{p(\left\{ x:d(x,S)\le\varepsilon\right\} )-p(S)}{\varepsilon}
\]
where $d(x,S)$ is the minimum Euclidean distance between $x$ and
$S$. The isoperimetric constant of $p$ (or Cheeger constant of $p$)
is the minimum possible ratio between the boundary measure of a subset
and the measure of the subset among all subsets of measure at most
half:
\[
\psi_{p}=\inf_{S\subseteq\R^{n}}\frac{p(\partial S)}{\min\left\{ p(S),p(\R^{n}\setminus S)\right\} }.
\]

For a Gaussian distribution and an unit hypercube in $\R^{n}$, this
ratio is a constant independent of the dimension, with the minimum
achieved by a halfspace as mentioned. Both of them belong to a much
more general class of probability distributions, called logconcave
distribution. A probability density function is logconcave if its
logarithm is concave along every line, i.e., for any $x,y\in\R^{n}$
and any $\lambda\in[0,1]$, 
\begin{equation}
f(\lambda x+(1-\lambda)y)\ge f(x)^{\lambda}f(y)^{1-\lambda}.\label{eq:logcon}
\end{equation}
Many common probability distributions are logconcave e.g., Gaussian,
exponential, logistic and gamma distributions. This also includes
indicator functions of convex sets, sets with the property that for
any two points $x,y\in K$, the line segment $[x,y]\subseteq K$. 

In the course of their study of algorithms for computing the volume,
in 1995, Kannan, Lovász and Simonovits made the following conjecture.
\begin{conjecture*}[KLS Conjecture \cite{KLS95}]
For any logconcave density p in $\R^{n}$,
\begin{equation}
\psi_{p}\ge c\cdot\inf_{\mathrm{halfspace}\ H}\frac{p(\partial H)}{\min\left\{ p(H),p(\R^{n}\setminus H)\right\} }.\label{eq:KLS}
\end{equation}
where $c$ is an absolute, universal constant independent of the dimension
and the density $p$.
\end{conjecture*}
For any halfspace $H$, its expansion is a one-dimensional quantity,
namely the expansion of the corresponding interval when the density
projected along the normal to the halfspace. Since projections of
logconcave densities are also logconcave (Lemma \ref{lem:marginal}),
it is not hard to argue that the isoperimetric ratio is $\Theta(1/\sigma_{f})$
for any one-dimensional logconcave density with variance $\sigma_{f}^{2}$.
This gives an explicit formula for the right hand side of (\ref{eq:KLS}).
Going forward, we use the notation $a\gtrsim b$ to denote $a\geq c\cdot b$
for some universal constant $c$ independent of the dimension and
all parameters under consideration.
\begin{lem}
For any $n$-dimensional logconcave density with covariance matrix
$A$,
\[
\inf_{\mathrm{halfspace}\ H}\frac{p(\partial H)}{\min\left\{ p(H),p(\R^{n}\setminus H)\right\} }\gtrsim\frac{1}{\sqrt{\norm A_{\op}}}
\]
where $\norm A_{\op}$ is the largest eigenvalue of $A$, or equivalently,
the spectral norm of $A$.
\end{lem}

It is therefore useful to consider the following normalization to
a given distribution: apply an affine transformation so that for the
transformed density $p$, we have $\E_{x\sim p}\left(X\right)=0$
and $\E_{x\sim p}\left(XX^{\top}\right)=I$, i.e., zero mean and identity
covariance; we call such a distribution \emph{isotropic. }For any
distribution with mean $\mu$ and covariance $A$, both well-defined,
we can apply the transformation $A^{-\frac{1}{2}}(X-\mu)$ to get
an isotropic distribution. Using this, we can reformulate the KLS
conjecture as follows:
\begin{conjecture*}[KLS, reformulated]
For any logconcave density p in $\R^{n}$ with covariance matrix
$A$, $\psi_{p}\gtrsim\norm A_{\op}^{-\frac{1}{2}}$. Equivalently,
$\psi_{p}\gtrsim1$ for any isotropic logconcave distribution $p$.
\end{conjecture*}

\subsection{Concentration of measure}

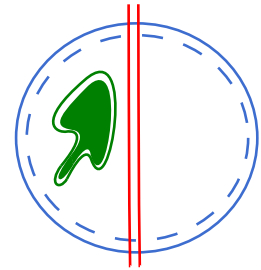
\begin{wrapfigure}{r}{0.25\textwidth}
\centering
\begin{tikzpicture}[y=0.80pt, x=0.80pt, yscale=-0.250000, xscale=0.250000, inner sep=0pt, outer sep=0pt]   \path[draw=c063cb8,dash pattern=on 8.00pt off 8.00pt,line cap=butt,miter     limit=4.00,draw opacity=0.761,nonzero rule,line width=1.000pt]     (355.9839,452.3622) ellipse (5.8864cm and 5.4603cm);   \path[draw=c063cb8,line cap=butt,miter limit=4.00,draw opacity=0.762,nonzero     rule,line width=1.000pt] (354.5455,452.3622) ellipse (6.4141cm and 6.0934cm);   \begin{scope}[cm={{1.00879,-0.01162,0.01706,0.68726,(-11.2869,120.82206)}}]     \path[draw=cff0000,line join=miter,line cap=butt,miter limit=4.00,even odd       rule,line width=1.000pt] (345.4545,120.5440) .. controls (334.0909,840.9986)       and (336.3636,840.9986) .. (336.3636,840.9986);     \path[draw=cff0000,line join=miter,line cap=butt,miter limit=4.00,even odd       rule,line width=1.000pt] (363.6364,120.5440) .. controls (352.2727,840.9986)       and (354.5455,840.9986) .. (354.5455,840.9986);   \end{scope}   \begin{scope}[cm={{0.9126,0.0,0.0,1.09343,(-144.56099,-66.27663)}},miter limit=4.00,line width=2.002pt]     \begin{scope}[cm={{1.00583,0.0,0.0,1.04013,(180.08503,10.57023)}},draw=c009200,miter limit=4.00,line width=1.000pt]       \begin{scope}[draw=c009200,miter limit=4.00,line width=1.000pt]         \path[draw=c009200,line join=miter,line cap=butt,miter limit=4.00,even odd           rule,line width=1.000pt] (209.0000,492.3622) .. controls (274.5268,428.2139)           and (221.5909,435.3338) .. (221.5909,435.3338) .. controls (198.8931,438.6397)           and (155.5607,420.3997) .. (256.8182,350.0895) .. controls (367.5714,269.6319)           and (308.7171,567.5409) .. (271.1893,482.2921) .. controls (258.3543,451.2350)           and (221.0000,514.3622) .. (221.0000,514.3622);         \path[draw=c009200,line join=miter,line cap=butt,miter limit=4.00,even odd           rule,line width=1.000pt] (209.0000,492.3622) .. controls (181.9482,521.9047)           and (204.9329,535.2203) .. (221.0000,514.3622);       \end{scope}     \end{scope}     \begin{scope}[cm={{0.75995,-0.09224,-0.00625,0.78339,(251.52894,138.36435)}},draw=c009200,fill=c009200,miter limit=4.00,line width=1.000pt]       \path[draw=c009200,fill=c009200,line join=miter,line cap=butt,miter         limit=4.00,even odd rule,line width=1.000pt] (215.2300,495.0927) .. controls         (160.4759,549.3920) and (185.8097,552.3336) .. (217.9904,510.7357);       \path[draw=c009200,fill=c009200,line join=miter,line cap=butt,miter         limit=4.00,even odd rule,line width=1.000pt] (215.2300,495.0927) .. controls         (269.9974,445.9938) and (221.5909,435.3338) .. (221.5909,435.3338) .. controls         (183.5914,424.4004) and (155.5607,420.3997) .. (256.8182,350.0895) .. controls         (367.5714,269.6319) and (302.6820,619.9579) .. (269.2876,491.6660) .. controls         (256.4526,460.6089) and (217.9904,510.7357) .. (217.9904,510.7357);     \end{scope}   \end{scope}
\end{tikzpicture}
\caption*{Concentration of measure}
\end{wrapfigure}

One motivation to study isoperimetry is the phenomenon known as concentration
of measure. This can be illustrated as follows: most of a Euclidean
unit ball in $\R^{n}$ lies within distance $O(\frac{1}{n})$ of its
boundary, and also within distance $O(\frac{1}{\sqrt{n}})$ of any
central hyperplane. Most of a Gaussian lies in an annulus of thickness
$O(1)$. For any subset of the sphere of measure $\frac{1}{2}$, the
measure of points at distance at least $\sqrt{\frac{\log n}{n}}$
from the set is a vanishing fraction. These concentration phenomena
are closely related to isoperimetry. For example, since the sphere
has good isoperimetry (about $\sqrt{n}$), the boundary of any subset
of measure $\frac{1}{2}$ is large, and summing up over all points
within distance $\frac{1}{\sqrt{n}}$ gives a constant fraction of
the entire sphere. 

The relationship between isoperimetry and concentration runs deep
with connections in both directions. In particular, the asymptotic
behavior as the dimension grows is of interest in both cases and we
will discuss this in more detail. We now review some basic definitions
and properties of convex sets and logconcave distributions.

For two subsets $A,B\subseteq\R^{n}$ their Minkowski sum is $A+B=\{x+y\,:\,x\in A,y\in B\}.$
The Brunn-Minkowski theorem says that if $A,B,A+B$ are measurable,
then 

\[
\vol(A+B)^{\frac{1}{n}}\geq\vol(A)^{\frac{1}{n}}+\vol(B)^{\frac{1}{n}}.
\]
For the cross-sections of a convex body $K$ orthogonal to a fixed
vector $u$, it says that the volume function $v(t)$ along any direction
is $\frac{1}{n-1}$-concave, i.e. $v(t)^{\frac{1}{n-1}}$ is concave.
If we replace each cross-section with a ball of the same volume, the
radius function is concave along $u$. 

A generalization of convex sets is logconcave functions (\ref{eq:logcon}).
Their basic properties are summarized by the following classical lemma. 
\begin{lem}[Dinghas; Prékopa; Leindler]
\label{lem:marginal} The product, minimum and convolution of two
logconcave functions is also logconcave; in particular, any linear
transformation or marginal of a logconcave density is logconcave.
\end{lem}

Unlike convex sets, the set of logconcave distributions is closed
under convolution. This is one of many reasons we work with this more
general class of functions. The proof of this follows from an analog
of the Brunn-Minkowski theorem for functions called the Prékopa-Leindler
inequality: Let $\lambda\in[0,1]$, and three bounded functions $f,g,h:\R^{n}\rightarrow\R_{+}$
satisfy $f(\lambda x+(1-\lambda)y)\ge g(x)^{\lambda}h(y)^{1-\lambda}$
for any $x,y\in\R^{n}$. Then, 
\[
\int_{\R^{n}}f\ge\left(\int_{\R^{n}}g\right)^{\lambda}\left(\int_{\R^{n}}h\right)^{1-\lambda}.
\]

\begin{figure}
\centering
\begin{tikzpicture}[y=0.80pt, x=0.80pt, yscale=-0.500000, xscale=0.500000, inner sep=0pt, outer sep=0pt] \path[draw=black,line join=miter,line cap=butt,miter limit=4.00,even odd   rule,line width=1.000pt] (73.4680,441.2501) -- (72.3420,440.3243) --   (34.0574,459.7645) .. controls (21.6712,503.2734) and (20.5452,502.3477) ..   (20.5452,502.3477) -- (61.0818,520.8621) -- (87.5433,478.7418); \path[draw=black,line join=miter,line cap=butt,miter limit=4.00,even odd   rule,line width=1.000pt] (87.5433,478.7418) -- (72.3420,440.3243) --   (72.3420,440.3243); \path[draw=black,line join=miter,line cap=butt,miter limit=4.00,even odd   rule,line width=1.000pt] (72.3420,440.3243) -- (202.9601,408.8498) --   (234.4886,443.1015); \path[draw=black,line join=miter,line cap=butt,miter limit=4.00,even odd   rule,line width=1.000pt] (202.9601,408.8498) .. controls (161.2975,462.5417)   and (161.2975,462.5417) .. (161.2975,462.5417); \path[draw=black,line join=miter,line cap=butt,miter limit=4.00,even odd   rule,line width=1.000pt] (161.8245,461.8129) -- (195.6354,501.4791) --   (234.4886,443.1015); \path[draw=black,line join=miter,line cap=butt,miter limit=4.00,even odd   rule,line width=1.000pt] (202.9601,408.8498) -- (304.3018,444.9530) --   (284.5965,489.3876); \path[draw=black,line join=miter,line cap=butt,miter limit=4.00,even odd   rule,line width=1.000pt] (304.7164,445.5690) -- (305.9908,470.8732); \path[draw=black,line join=miter,line cap=butt,miter limit=4.00,even odd   rule,line width=1.000pt] (285.1595,489.3876) -- (305.9908,470.8732); \path[draw=black,line join=miter,line cap=butt,miter limit=4.00,even odd   rule,line width=1.000pt] (61.0818,520.8621) -- (284.5965,489.3876); \path[draw=black,dash pattern=on 4.08pt off 4.08pt,line join=miter,line   cap=butt,miter limit=4.00,even odd rule,line width=1.000pt] (20.5452,502.3477)   .. controls (162.4235,460.6902) and (162.4235,461.6159) .. (162.4235,461.6159)   -- (227.7325,474.5760) -- (285.1595,489.3876); \path[draw=black,line join=miter,line cap=butt,miter limit=4.00,even odd   rule,line width=1.000pt] (87.5433,478.7418) -- (234.4886,443.1015) --   (305.9908,470.8732); \path[cm={{0.98776,-0.15598,0.03615,0.99935,(0.0,0.0)}},draw=black,line   cap=butt,miter limit=4.00,nonzero rule,line width=1.000pt] (442.7515,536.5011)   ellipse (0.4042cm and 0.9496cm); \path[cm={{0.96214,-0.27254,0.04782,0.99886,(0.0,0.0)}},draw=black,line   cap=butt,miter limit=4.00,nonzero rule,line width=1.000pt] (624.7074,638.6868)   ellipse (0.5872cm and 1.3652cm); \path[cm={{0.98525,-0.1711,0.03288,0.99946,(0.0,0.0)}},draw=black,line   cap=butt,miter limit=4.00,nonzero rule,line width=1.000pt] (712.9545,589.9141)   ellipse (0.3593cm and 0.6593cm); \path[draw=black,line join=miter,line cap=butt,even odd rule,line width=1.000pt]   (454.4216,433.5888) .. controls (531.5548,407.0616) and (597.9580,399.8869) ..   (724.1319,444.2449); \path[draw=black,line join=miter,line cap=butt,even odd rule,line width=1.000pt]   (724.5652,490.4161) .. controls (574.6907,554.2867) and (454.7745,500.2673) ..   (454.7745,500.2673); \path[draw=c063cb8,line join=miter,line cap=butt,miter limit=4.00,draw   opacity=0.761,even odd rule,line width=1.000pt] (457.8599,471.2212) ..   controls (720.2251,471.2212) and (720.2251,471.2212) .. (720.2251,471.2212); \begin{scope}[cm={{0.00882,-0.67689,0.90317,0.00661,(-235.29315,707.10301)}}]   \path[draw=black,line join=miter,line cap=butt,even odd rule,line width=1.000pt]     (355.3734,627.2486) -- (355.3734,708.9435);   \path[draw=black,line join=miter,line cap=butt,even odd rule,line width=1.000pt]     (395.5272,677.3068) -- (354.6012,709.1438);   \path[draw=black,line join=miter,line cap=butt,even odd rule,line width=1.000pt]     (314.1952,676.6060) -- (355.1212,708.4430); \end{scope}
\end{tikzpicture}
\caption*{Brunn-Minkowski applied to convex bodies: the radius function is concave}
\end{figure}

\subsection{The original motivation: A sampling algorithm}

The study of algorithms has rich connections to high-dimensional convex
geometry. It was the study of an algorithm for sampling that led to
the KLS conjecture.

\subsubsection{Model of computation}

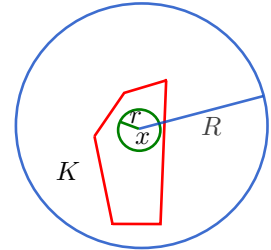
\begin{wrapfigure}{r}{0.25\textwidth}
\centering
\begin{tikzpicture}[y=0.80pt, x=0.80pt, yscale=-0.250000, xscale=0.250000, inner sep=0pt, outer sep=0pt]   \path[draw=c063cb8,line cap=butt,miter limit=4.00,draw opacity=0.761,nonzero     rule,line width=1.000pt] (383.3333,438.4733) ellipse (6.6897cm and 6.5068cm);   \path[draw=c008000,line cap=butt,miter limit=4.00,nonzero rule,line     width=1.000pt] (378.7037,446.8067) ellipse (1.1237cm and 1.0975cm);   \path[draw=cff0000,line join=miter,line cap=butt,miter limit=4.00,even odd     rule,line width=1.000pt] (294.0000,458.3622) -- (350.2317,376.3788) ..     controls (431.7737,352.3586) and (429.6279,352.3586) .. (429.6279,352.3586) --     (418.8987,624.5880) -- (328.0000,624.3622);   \path[draw=cff0000,line join=miter,line cap=butt,miter limit=4.00,even odd     rule,line width=1.000pt] (294.0000,458.3622) .. controls (328.6491,626.5891)     and (328.0000,624.3622) .. (328.0000,624.3622);   \path[draw=c063cb8,line join=miter,line cap=butt,miter limit=4.00,draw     opacity=0.761,even odd rule,line width=1.000pt] (378.0000,444.3622) --     (613.4085,382.3622);   \path[draw=c008000,fill=c008000,line join=miter,line cap=butt,miter     limit=4.00,even odd rule,line width=1.000pt] (342.9577,431.2354) --     (378.0000,444.3622);   \path[fill=black,line join=miter,line cap=butt,line width=1.000pt]     (220.9718,541.0946) node[above right] (text5403) {$K$};   \path[xscale=1.012,yscale=0.988,fill=c063cb8,line join=miter,line cap=butt,fill     opacity=0.761,line width=1.000pt] (490.1948,460.6454) node[above right]     (text5407) {$R$};   \path[fill=black,line join=miter,line cap=butt,line width=1.000pt]     (373.2394,455.8833) node[above right] (text5411) {};   \path[fill=black,line join=miter,line cap=butt,line width=1.000pt]     (384.3478,455.8405) node[above right] (text4764) {};   \path[xscale=1.044,yscale=0.957,fill=c008000,line join=miter,line cap=butt,miter     limit=4.00,line width=1.000pt] (355.3322,495) node[above right]     (text4768) {$x$};   \path[fill=black,line join=miter,line cap=butt,miter limit=4.00,line     width=1.000pt] (360.2473,434.0002) node[above right] (text4772) {$r$};
\end{tikzpicture}
\caption*{Well-guaranteed oracle}
\end{wrapfigure}

For algorithmic problems such as sampling, optimization and integration,
the following general model of computation is standard. Convex bodies
and logconcave functions are presented by \emph{well-guaranteed} oracles.
For a convex body $K\subset\R^{n}$, a well-guaranteed membership
oracle is given by numbers $R\ge r>0$, a point $x_{0}\in K$ with
the guarantee that $x_{0}+rB_{n}\subseteq K\subseteq RB_{n}$ and
an oracle that answers YES or NO to a query of the form ``$x\in K$?''
for any $x\in\R^{n}$. Another oracle of interest is a well-guaranteed
\emph{separation} oracle; it is given by the same parameters $r,R$
(but no starting point in the set), and a stronger oracle: for a query
``$x\in K$?'', the oracle either answers YES or answers NO and
provides a hyperplane that separates $x$ from $K$. For an integrable
function $f:\R^{n}\rightarrow\R_{+}$, a well-guaranteed function
oracle is given by a bound on the norm of its center of gravity, upper
and lower bounds on the eigenvalues of the covariance matrix with
density proportional to $f$ and an oracle that returns $f(x)$ for
any $x\in\R^{n}$. The complexity of an algorithm is measured first
in terms of the number of calls to the oracle and second in terms
of the total number of arithmetic operations performed. An algorithm
is considered efficient (and a problem \emph{tractable}) if its complexity
is bounded by a fixed polynomial in the dimension and other input
parameters, which typically include an error parameter and a probability
of failure parameter. We use $\widetilde{O}(\cdot)$ to suppress logarithmic
factors in the leading expression and $O^{*}(\cdot)$ to suppress
logarithmic factors as well as dependence on error parameters. 

\subsubsection{Sampling with a Markov chain}

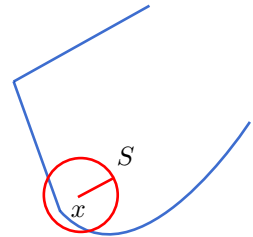
\begin{wrapfigure}{r}{0.25\textwidth}
\centering
\begin{tikzpicture}[y=0.80pt, x=0.80pt, yscale=-0.6, xscale=0.6, inner sep=0pt, outer sep=0pt]   \path[draw=c063cb8,line join=miter,line cap=butt,miter limit=4.00,draw     opacity=0.761,even odd rule,line width=1.000pt] (171.4626,414.0940) --     (278.5744,354.2415);   \path[draw=c063cb8,line join=miter,line cap=butt,miter limit=4.00,draw     opacity=0.761,even odd rule,line width=1.000pt] (171.5008,413.9242) --     (208.2029,516.3927);   \path[draw=c063cb8,line join=miter,line cap=butt,miter limit=4.00,draw     opacity=0.761,even odd rule,line width=1.000pt] (207.9736,515.9397) ..     controls (269.2629,583.3207) and (357.6946,445.8634) .. (357.6946,445.8634);   \path[draw=cff0000,line cap=butt,miter limit=4.00,fill opacity=0.761,nonzero     rule,line width=1.000pt] (224.4444,503.5104) circle (0.8184cm);   \path[draw=cff0000,fill=cff0000,line join=miter,line cap=butt,miter     limit=4.00,even odd rule,line width=1.000pt] (222.2222,505.1400) .. controls     (250.0000,490.3252) and (250.0000,490.3252) .. (250.0000,490.3252);   \path[fill=black,line join=miter,line cap=butt,line width=1.000pt]     (252.7778,480.1029) node[above right] (text4702) {$S$};   \path[fill=black,line join=miter,line cap=butt,line width=1.000pt]     (216.6667,520.8437) node[above right] (text4706) {$x$};
\end{tikzpicture}
\caption*{Ball walk}
\end{wrapfigure}

An important problem in the theory of algorithms is efficiently sampling
high-dimensional sets and distributions. As we will presently see,
sampling is closely related to an even more basic and ancient problem:
estimating the volume (or integral). 

Algorithms for sampling are based on Markov chains whose stationary
distribution is the target distribution for the sampling problem.
One such method is the \emph{ball walk} \cite{L90} for sampling from
the uniform distribution over a convex body (compact convex set),
a particular discretization of Brownian motion. Start with some point
in the body. At a current point $x$, 
\begin{enumerate}
\item Pick a random point $y$ in the ball of a fixed radius $\delta$ centered
at $x$. 
\item If $y$ is in the body, go to $y$, otherwise stay at $x$. 
\end{enumerate}
This process converges to the uniform distribution over any compact,
``well-connected'' domain. But how efficient is it? To answer this
question, we have to study its rate of convergence. This is a subject
on its own with several general tools. 

One way to bound the rate of convergence is via the smallest non-zero
eigenvalue of the transition operator. If this eigenvalue is $\lambda$,
then an appropriate notion of distance from the current distribution
to the stationary distribution decreases by a multiplicative factor
of $\left(1-\lambda\right)^{t}$ after $t$ steps of a discrete-time
Markov chain. Thus, the larger the spectral gap, the more rapid the
convergence. 
\begin{itemize}
\item But how to bound the spectral gap? There are many methods. In the
context of high-dimensional continuous distributions, where the state
space is not finite, one useful method is the notion of \emph{conductance}.
This can be viewed as the isoperimetry or expansion of the state space
under the transition operator. More precisely, for a Markov chain
defined by a state space $\Omega$ with transition operation $p(x\rightarrow y$)
and stationary distribution $Q$, the conductance of a measurable
subset is
\end{itemize}
\[
\phi(S)\defeq\frac{\int_{y\notin S}\int_{x\in S}P(x\rightarrow y)Q(x)dxdy}{\min\left\{ Q(S),Q(\Omega\setminus S)\right\} }
\]
and the conductance of the entire Markov chain is 
\[
\phi\defeq\inf_{S\subset\Omega}\phi(S).
\]
The following general theorem, due to Jerrum and Sinclair \cite{SJ89}
was extended to the continuous setting by Lovász and Simonivits \cite{LS93}.
\begin{thm}[\cite{LS93}]
\label{thm:cheeger}Let $Q_{t}$ be the distribution of the current
point after $t$ steps of a Markov chain with stationary distribution
$Q$ and conductance at least $\phi,$ starting from initial distribution
$Q_{0}.$ Then, with $M=\sup_{A}\frac{Q_{0}(A)}{Q(A)}$,
\[
d_{TV}(Q_{t},Q)\le\sqrt{M}\left(1-\frac{\phi^{2}}{2}\right)^{t}
\]
where $d_{TV}(Q_{t},Q)$ is the total variation distance between $Q_{t}$
and $Q$.
\end{thm}

The mixing time of a Markov chain is related to its conductance via
the following fundamental inequality.

\begin{equation}
\frac{1}{\phi}\lesssim\tau\lesssim\frac{1}{\phi^{2}}\log M\label{eq:JS}
\end{equation}

The conductance of the ball walk can be lower bounded by the product
of two parameters. The first parameter is a mix of probability and
geometry and asks for the minimum distance such that two points at
this distance will have some constant overlap in their next-step distributions.
The second is a purely geometry question, it is exactly the Cheeger
constant of the stationary distribution. It turns out the first parameter
can be estimated (for a precise statement, see Theorem \ref{thm:KLS-mixing}),
thereby reducing the problem of bounding the conductance of the Markov
chain to bounding the Cheeger constant of the stationary distribution.
This motivated the conjecture that the Cheeger constant for any isotropic
logconcave density is in fact at least a universal constant, independent
of the density and the dimension. If true, it would imply a bound
of $O^{*}(n^{2})$ on the mixing time of the ball walk from a warm
start in an isotropic convex body, which is the best possible bound
(it is tight for the isotropic hypercube).

\section{Connections }

Here we discuss a wide array of implications of the conjecture, in
geometry, probability and algorithms. For further details, we refer
the reader to recent books on the topic \cite{brazitikos2014geometry,alonso2015approaching,ArtsteinGM2015}. 

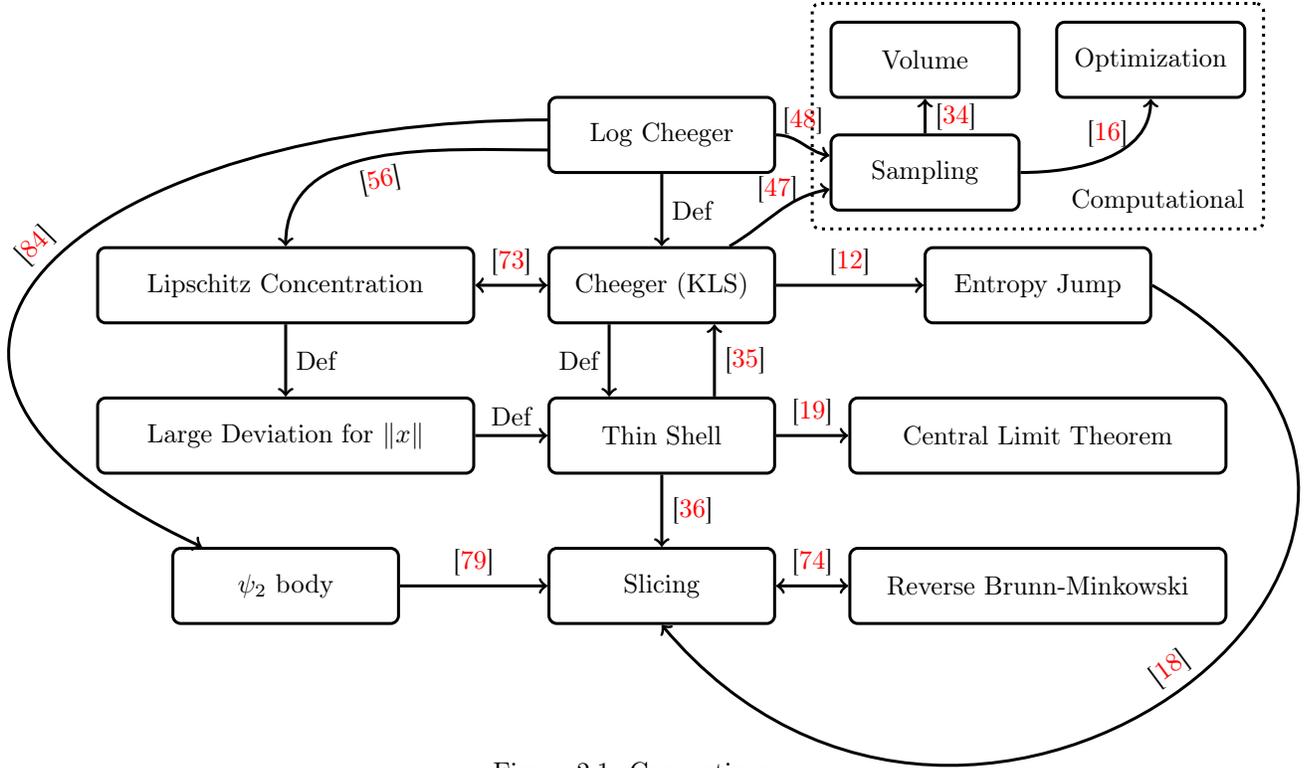
\begin{figure}
\hspace*{-3cm}
\begin{tikzpicture}
[thinbox/.style={rectangle,draw=black,line width=0.4mm,minimum width=3cm,minimum height=1cm,rounded corners=3pt}, 
widebox/.style={rectangle,draw=black,line width=0.4mm,minimum width=5cm,minimum height=1cm,rounded corners=3pt},
dottedbox/.style={rectangle,dotted,draw=black,line width=0.4mm,minimum width=6cm,minimum height=3cm,rounded corners=3pt,align=right,text width=5.5cm,text height=2.5cm}, 
tinybox/.style={rectangle,draw=black,line width=0.4mm,minimum width=2.5cm,minimum height=1cm,rounded corners=3pt}
]

\node at (0,0) (slicing) [thinbox] {Slicing};
\node at (0,2) (thinshell) [thinbox] {Thin Shell};
\node at (0,4) (kls) [thinbox] {Cheeger (KLS)};
\node at (0,6) (logch) [thinbox] {Log Cheeger};
\node at (5,6.25) [dottedbox] {Computational};
\node at (3.5,5.5) (sampling) [tinybox] {Sampling};
\node at (3.5,7) (volume) [tinybox] {Volume};
\node at (6.5,7) (optimization) [tinybox] {Optimization};

\draw [->,line width=0.4mm] (thinshell) -- node[right] {\cite{EldanK2011}} (slicing);
\draw [->,line width=0.4mm] ([xshift=-0.7cm]kls.south) -- node[left] {Def} ([xshift=-0.7cm]thinshell.north);
\draw [<-,line width=0.4mm] ([xshift=0.7cm]kls.south) -- node[right] {\cite{Eldan2013}} ([xshift=0.7cm]thinshell.north);
\draw [->,line width=0.4mm] (logch) -- node[right] {Def} (kls);

\draw [->,line width=0.4mm] (logch) to[out=0,in=170] node[above] {\cite{KannanLM06}} (sampling);
\draw [->,line width=0.4mm] (kls) to[out=30,in=190] node[above] {\cite{KLS97}} (sampling);
\draw [->,line width=0.4mm] (sampling) -- node[right] {\cite{DyerFK89}} (volume);
\draw [->,line width=0.4mm] (sampling) to[out=0,in=-90] node[above] {\cite{BV04}} (optimization);

\node at (-5,0) (psi) [thinbox] {$\psi_2$ body};
\node at (-5,2) (deviation) [widebox] {Large Deviation for $\|x\|$}; \node at (-5,4) (lip) [widebox] {Lipschitz Concentration};
\node at (5,0) (brunn) [widebox] {Reverse Brunn-Minkowski};
\node at (5,2) (clt) [widebox] {Central Limit Theorem};
\node at (5,4) (jump) [thinbox] {Entropy Jump};

\draw [->,line width=0.4mm] (psi) -- node[above] {\cite{paouris2012small}} (slicing);
\draw [->,line width=0.4mm] (deviation) -- node[above] {Def} (thinshell);
\draw [<->,line width=0.4mm]  (lip) -- node[above]{\cite{Milman2009}} (kls) ;
\draw [->,line width=0.4mm] (lip) -- node[right] {Def} (deviation);

\draw [->,line width=0.4mm] (kls) -- node[above] {\cite{ball2012entropy}} (jump);
\draw [->,line width=0.4mm] (thinshell) -- node[above] {\cite{bobkov2003concentration}} (clt);
\draw [<->,line width=0.4mm]  (slicing) -- node[above] {\cite{milman1989isotropic}}  (brunn) ;
\draw [->,line width=0.4mm]  (jump.east) to[out=-30,in=-50,distance=6.5cm]  node[above,sloped] {\cite{bobkov2010entropy}}  (slicing.south) ;
\draw [->,line width=0.4mm]  ([yshift=-0.2cm]logch.west) to[out=180,in=90]  node[below,sloped] {\cite{ledoux1999concentration}}  (lip) ;
\draw [->,line width=0.4mm]  ([yshift=0.2cm]logch.west) to[out=180,in=155,distance=6cm]  node[above,sloped] {\cite{stavrakakis2013geometry}}  (psi) ; 

\end{tikzpicture}
\vspace*{-4cm}

\caption{Connections}

\end{figure}

\subsection{Geometry and Probability }

The KLS conjecture implies the slicing conjecture and the thin-shell
conjecture. Each of these has powerful and surprising consequences.
We discuss them in order of the strength of the conjectures \textemdash{}
slicing, thin-shell, KLS.

\subsubsection{Slicing to anti-concentration}

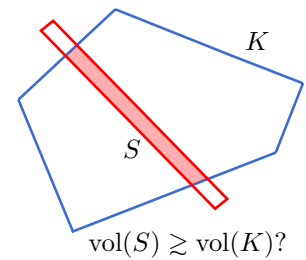
\begin{wrapfigure}{r}{0.25\textwidth}
\centering
\begin{tikzpicture}[y=0.80pt, x=0.80pt, yscale=-0.3000000, xscale=0.3000000, inner sep=0pt, outer sep=0pt]   \path[draw=c063cb8,line join=miter,line cap=butt,miter limit=4.00,draw     opacity=0.761,even odd rule,line width=1.000pt] (89.3829,473.8448) --     (175.3188,681.7452) -- (495.0000,557.3622);   \path[draw=c063cb8,line join=miter,line cap=butt,miter limit=4.00,draw     opacity=0.761,even odd rule,line width=1.000pt] (495.0000,557.3622) --     (537.9999,447.9389);   \path[draw=c063cb8,line join=miter,line cap=butt,miter limit=4.00,draw     opacity=0.761,even odd rule,line width=1.000pt] (89.3829,473.8448) --     (242.2085,331.3622);   \path[draw=c063cb8,line join=miter,line cap=butt,miter limit=4.00,draw     opacity=0.761,even odd rule,line width=1.000pt] (242.2085,331.3622) ..     controls (533.8042,447.9601) and (537.9999,447.9389) .. (537.9999,447.9389);   \path[cm={{0.7791,-0.6269,0.69866,0.71545,(0.0,0.0)}},draw=cff0000,line     cap=butt,miter limit=4.00,fill opacity=0.761,nonzero rule,line     width=1.000pt,rounded corners=0.0000cm] (-166.3427,364.3955) rectangle     (-141.6752,757.1052);   \path[fill=black,line join=miter,line cap=butt,line width=1.000pt]     (200.0334,723.5007) node[above right] (text5305) {$\mathrm{vol}(S) \apprge \mathrm{vol}(K)$?};   \path[fill=black,line join=miter,line cap=butt,line width=1.000pt]     (446.5659,395.5517) node[above right] (text4142) {$K$};   \path[fill=cff0000,fill opacity=0.316] (265.2715,507.4381) .. controls     (211.4740,452.3657) and (166.9055,406.6835) .. (166.2304,405.9221) --     (165.0030,404.5378) -- (173.2521,396.8505) .. controls (177.7892,392.6224) and     (181.5887,389.2011) .. (181.6955,389.2475) .. controls (182.0090,389.3836) and     (386.3381,598.6835) .. (386.2313,598.7591) .. controls (385.9773,598.9388) and     (363.5273,607.5756) .. (363.3213,607.5729) .. controls (363.1915,607.5709) and     (319.0691,562.5105) .. (265.2715,507.4381) -- cycle;   \path[fill=black,line join=miter,line cap=butt,line width=1.000pt]     (254.6037,566.1661) node[above right] (text4142-9) {$S$};
\end{tikzpicture}
\caption*{Slicing Conjecture}
\end{wrapfigure}

The slicing conjecture (a.k.a. the hyperplane conjecture) is one of
the main open questions in convex geometry. It is implied by KLS conjecture.
Ball first showed that a positive resolution of the KLS conjecture
implies the slicing conjecture \cite{Ball2006lecture}. Eldan and
Klartag \cite{EldanK2011} later gave a more refined quantitative
relation (Theorem \ref{thm:Eldan-Klartag}).

The conjecture says that any convex body in $\R^{n}$ of unit volume
has a hyperplane section whose $(n-1)$-dimensional volume is at least
a universal constant. Ball \cite{ball1988logarithmically} gave the
following equivalent conjecture for logconcave distributions.
\begin{conjecture}[Slicing conjecture \cite{Bourgain1986,ball1988logarithmically}]
For any isotropic logconcave density $p$ in $\R^{n}$, the isotropic
(slicing) constant $L_{p}\defeq p(0)^{\frac{1}{n}}$ is $O(1)$.
\end{conjecture}

Geometrically, this conjecture says that an isotropic logconcave distribution
cannot have much mass around the origin. The best known bound is $L_{p}\lesssim n^{\frac{1}{4}}$
\cite{Klartag2006,Bourgain1986}.

Paouris showed that if the slicing conjecture is true, then a logconcave
distribution satisfies a strong anti-concentration property (small
ball probability). 
\begin{thm}[Small ball probability \cite{paouris2012small}]
If slicing conjecture is true, for any isotropic logconcave density
$p$ in $\R^{n}$, we have

\[
\P_{x\sim p}\left(\norm x\leq t\sqrt{n}\right)=O(t)^{n}
\]
for all $0\leq t\leq c$ for some universal constant $c$. 
\end{thm}

Paouris also showed that the inequality holds unconditionally with
exponent $O(\sqrt{n})$ \cite{paouris2012small}.

Another nice application of anti-concentration is to lower bound the
entropy of a distribution.
\begin{thm}[Entropy of logconcave distribution \cite{bobkov2010entropy}]
If the slicing conjecture is true, for any isotropic logconcave density
$p$ in $\R^{n}$, we have

\[
-O(n)\leq\E_{x\sim p}\log\frac{1}{p(x)}\leq O(n).
\]
\end{thm}

The Brunn-Minkowski inequality is not tight when applied to convex
bodies that are large in different directions (e.g., $K=\{\varepsilon x^{2}+\varepsilon^{-1}y^{2}\leq1\}$
and $T=\{\varepsilon^{-1}x^{2}+\varepsilon y^{2}\leq1\}$ with tiny
$\varepsilon$). The anti-concentration aspect of slicing also shows
the following reverse inequality.
\begin{thm}[Reverse Brunn-Minkowski Inequalities \cite{milman1989isotropic}]
If slicing conjecture is true, then for any isotropic convex sets
$K$ and $T$, we have
\[
\text{vol}(K+T)^{1/n}\leq O(1)\left(\text{vol}(K)^{1/n}+\text{vol}(T)^{1/n}\right).
\]
\end{thm}

We remark that all these consequences of the slicing conjecture are
in fact equivalent to the conjecture itself \cite{dafnis2010small,bobkov2010entropy,bourgain2004symmetrization}. 

\subsubsection{Thin shell to central limit theorem}

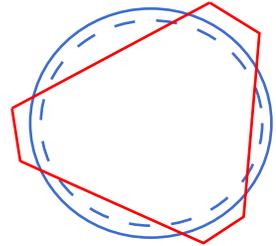
\begin{wrapfigure}{r}{0.25\textwidth}
\centering
\begin{tikzpicture}[y=0.80pt, x=0.80pt, yscale=-0.250000, xscale=0.250000, inner sep=0pt, outer sep=0pt]   \path[draw=c063cb8,dash pattern=on 8.00pt off 8.00pt,line cap=butt,miter     limit=4.00,draw opacity=0.761,nonzero rule,line width=1.000pt]     (355.9839,452.3622) ellipse (5.8864cm and 5.4603cm);   \path[draw=c063cb8,line cap=butt,miter limit=4.00,draw opacity=0.762,nonzero     rule,line width=1.000pt] (354.5455,452.3622) ellipse (6.4141cm and 6.0934cm);   \path[draw=cff0000,line join=miter,line cap=butt,miter limit=4.00,even odd     rule,line width=1.000pt] (92.0000,424.3622) -- (466.0000,224.3622);   \path[draw=cff0000,line join=miter,line cap=butt,miter limit=4.00,even odd     rule,line width=1.000pt] (466.0000,224.3622) -- (560.0000,282.3622) ..     controls (530.4225,628.8411) and (530.0000,629.8270) .. (530.0000,629.8270) ..     controls (455.3521,679.1228) and (454.0000,679.3622) .. (454.0000,679.3622) --     (108.0000,524.3622);   \path[draw=cff0000,line join=miter,line cap=butt,miter limit=4.00,even odd     rule,line width=1.000pt] (92.0000,424.3622) .. controls (107.4930,525.7706)     and (108.0000,524.3622) .. (108.0000,524.3622);
\end{tikzpicture}
\caption*{Thin Shell conjecture}
\end{wrapfigure}

The Central Limit Theorem says that a random marginal of a hypercube
is approximately Gaussian. Brehm and Voigt asked if the same is true
for convex sets. Anttila, Ball and Perissinaki \cite{Anttila2003}
observed that this is true for any distribution on the sphere. Therefore,
this also holds for any distributions with the norm $\norm x$ concentrated
at some value. Here, we state the version by Bobkov that holds for
any distribution.
\begin{thm}[Central Limit Theorem \cite{bobkov2003concentration}]
Let $\mu$ be an isotropic probability on $\Rn$, which might not
be logconcave. Assume that 
\begin{equation}
\mu\left(\left|\frac{\norm x_{2}}{\sqrt{n}}-1\right|\geq\varepsilon\right)\leq\varepsilon\label{eq:thin_shell_weak}
\end{equation}
for some $0<\varepsilon<1/3$. Let $g_{\theta}(s)=\mu(\{x^{\top}\theta=s\})$
and $g(s)=\frac{1}{\sqrt{2\pi}}\exp(-\frac{s^{2}}{2})$. Then, for
every $\delta>0$, we have that 
\[
\P\left(\left\{ \theta\in S^{n-1}:\text{ }\max_{t\in\R}\left|\int_{-\infty}^{t}g_{\theta}(s)ds-\int_{-\infty}^{t}g(s)ds\right|\geq2\delta+\frac{6}{\sqrt{n}}+4\varepsilon\right\} \right)\le c_{1}\delta^{-\frac{3}{2}}\exp(-c_{2}\delta^{4}n)
\]
for some universal constants $c_{1},c_{2}>0$.
\end{thm}

So, it suffices to prove that $\norm x$ is concentrated near $\sqrt{n}$
for any isotropic logconcave distribution. In a seminal work, Klartag
proved that (\ref{eq:thin_shell_weak}) holds with $\varepsilon\lesssim\log^{-\frac{1}{2}}n$
\cite{Klartag2007}. Shortly after, Fleury, Guédon and Paouris gave
an alternative proof with $\varepsilon\lesssim\log^{-1/6}n\cdot(\log\log n)^{2}$
\cite{fleury2007stability}. It is natural to ask for the optimal
bound for (\ref{eq:thin_shell_weak}). If the KLS conjecture is true,
we can apply the conjecture for the sets $\{\norm x_{2}\leq r\}$
for different values of $r$ and get the following conjecture, suggested
in \cite{Anttila2003,Bobkov2003}. 
\begin{conjecture}[Thin shell conjecture]
For any isotropic logconcave density $p$ in $\R^{n}$, the thin
shell constant $\sigma_{p}\defeq\Var_{x\sim p}\norm X_{2}$ is $O(1)$.
\end{conjecture}

The relation $\sigma_{p}\lesssim\psi_{p}^{-1}$ is an exercise. Eldan
and Klartag showed that the slicing constant is bounded by the thin-shell
constant (to within a universal constant).
\begin{thm}[\cite{EldanK2011}]
\label{thm:Eldan-Klartag} $L_{p}\lesssim\sup_{p}\sigma_{p}$ where
the maximization is over all isotropic logconcave distribution. 
\end{thm}

Since Klartag's bound on $\sigma_{p}$, there has been much effort
to improve the bound (see Table \ref{tab:thinshell}). In a breakthrough,
Eldan \cite{Eldan2013} showed that the thin shell conjecture is in
fact equivalent to the KLS conjecture up a logarithmic factor (see
Theorem \ref{thm:eldan}).

\begin{table}[h]
\centering{}%
\begin{tabular}{|c|c|}
\hline 
Year/Authors  & $\sigma_{p}$\tabularnewline
\hline 
\hline 
2006/Klartag \cite{Klartag2007} & $\sqrt{n}/\sqrt{\log n}$\tabularnewline
\hline 
2006/Fleury-Guédon-Paouris \cite{fleury2007stability} & $\sqrt{n}\frac{(\log\log n)^{2}}{\log^{1/6}n}$\tabularnewline
\hline 
2006/Klartag \cite{Klartag2007b} & $n^{4/10}$\tabularnewline
\hline 
2010/Fleury \cite{Fleury2010} & $n^{3/8}$\tabularnewline
\hline 
2011/Guedon-Milman \cite{GuedonM11} & $n^{1/3}$\tabularnewline
\hline 
2016/Lee-Vempala \cite{LeeV17KLS} & $n^{1/4}$\tabularnewline
\hline 
\end{tabular}\caption{\label{tab:thinshell} Progress on the thin shell bound.}
\end{table}

\subsubsection{Isoperimetry to concentration}

After discussing two conjectures that are potentially weaker than
KLS conjecture, we now move to some implications of the KLS conjecture
itself. 

The Poincaré constant of a measure is the minimum possible ratio of
the expected squared gradient to the variance over smooth functions.
Applying Cheeger's inequality \cite{Cheeger69} this constant is at
least the square of the Cheeger constant. The reverse inequality holds
for logconcave measures and was proved by Buser \cite{buser1982note}
(Ledoux \cite{ledoux1994simple} gave another proof).
\begin{thm}[Poincaré inequality \cite{Mazja60,Cheeger69,buser1982note,ledoux1994simple}]
 For any isotropic logconcave density $p$ in $\R^{n}$, we have
that
\[
\zeta_{p}\defeq\inf_{\mathrm{smooth}\ g}\frac{\E_{p}\left(\norm{\nabla g(x)}_{2}^{2}\right)}{\Var_{p}\left(g(x)\right)}\approx\psi_{p}^{2}\cdot
\]
for any smooth $g$.
\end{thm}

The Poincaré inequality is important in the study of partial differential
equations. For example, the Poincaré constant governs exactly how
fast the heat equation converges. For the logconcave setting, the
choice of $\ell_{2}$ norm is not important and it can be generalized
as follows:
\begin{thm}[Generalized Poincaré inequality \cite{Milman2009}]
For any isotropic logconcave density $p$ in $\R^{n}$ and for all
$1\leq q<\infty$, we have that
\[
\E_{x\sim p}\left|\nabla g(x)\right|\apprge\frac{\psi_{p}}{q}\cdot\left(\E_{x\sim p}\left|g(x)-\E_{y\sim p}g(y)\right|^{q}\right)^{1/q}
\]
for any smooth $g$.
\end{thm}

Together with previous inequalities, we can summarize the relationships
as follows: for any isotropic logconcave density $p$ in $\R^{n}$,
\[
L_{p}\lesssim\sup_{p}\sigma_{p}\text{ \quad and\quad\ }\sigma_{p}\lesssim\frac{1}{\sqrt{\zeta_{p}}}\approx\frac{1}{\psi_{p}}
\]
where the relation ``$\lesssim$'' hides only universal constants
independent of the density $p$ and the dimension $n$.

We next turn to concentration inequalities. The classical concentration
theorem of Levy says that any Lipschitz function $g$ on the sphere
in $\R^{n}$ is concentrated near its mean (or median): 
\[
\P\left(|g(x)-\E_{y}g(y)|\ge t\right)\le2e^{-\frac{1}{2}t^{2}n}.
\]
The following theorem is an analogous statement for any logconcave
density, and is due to Gromov and Milman. 
\begin{thm}[Lipschitz concentration \cite{GromovM83}]
For any $L$-Lipschitz function $g$ in $\R^{n},$ and isotropic
logconcave density $p$,
\[
\P_{x\sim p}\left(\left|g(x)-\E g\right|>L\cdot t\right)=e^{-\Omega(t\psi_{p})}.
\]
\end{thm}

Milman \cite{Milman2009} showed the reverse, namely that a Lipschtiz
concentration inequality implies a lower bound on the Cheeger constant. 

The next consequence is information-theoretic. Let $X$ be a random
variable from an $n$-dimensional distribution with a density $p$.
Its entropy is $\Ent(X)=-\E(\log p)=-\int_{\R^{n}}p(x)\log p(x)\,dx$.
The Shannon-Stam inequality says that for independent random vectors
$X,Y\sim p$, 
\[
\Ent\left(\frac{X+Y}{\sqrt{2}}\right)\ge\Ent(X)
\]
with equality only if $p$ is Gaussian. Quantifying the increase in
entropy has been a subject of investigation. Ball and Nguyen \cite{ball2012entropy}
proved the following bound on the entropy gap.
\begin{thm}[Entropy jump \cite{ball2012entropy}]
 Let $p$ be an isotropic logconcave density and $X,Y\sim p$ in
$\R^{n}$ and $Z$ be drawn from a standard Gaussian in $\R^{n}$.
Then, 
\[
\Ent\left(\frac{X+Y}{\sqrt{2}}\right)-\Ent(X)\gtrsim\psi_{p}^{2}\left(\Ent(Z)-\Ent(X)\right).
\]
\end{thm}

\subsection{Algorithms}

In this section we discuss algorithmic connections of the KLS conjecture.

\subsubsection{Sampling}

The ball walk can be used to sample from any density using a Metropolis
filter. To sample from the density $Q$, we repeat the following:
at a point $x$, 
\begin{enumerate}
\item Pick a random point $y$ in the $\delta$-ball centered at $x$. 
\item Go to $y$ with probability $\min\left\{ 1,\frac{Q(y)}{Q(x)}\right\} $. 
\end{enumerate}
If the resulting Markov chain is ergodic, the current distribution
approaches a unique stationary distribution. The complexity of sampling
depends on the \emph{rate }of convergence to stationarity. For an
isotropic logconcave distribution, the rate of convergence of the
ball walk from a \emph{warm} starting distribution is bounded in terms
of the Cheeger constant. A starting distribution $Q_{0}$ is \emph{warm}
with respect to the stationary distribution $Q$ if $\E_{x\sim Q}\frac{Q_{0}(x)}{Q(x)}$
is bounded by a constant.
\begin{thm}[\cite{KLS97}]
\label{thm:KLS-mixing} For an isotropic logconcave density $p$,
the ball walk mixes from a warm start in $O^{*}\left(n^{2}/\psi_{p}^{2}\right)$
steps.
\end{thm}

\subsubsection{Sampling to Convex Optimization}

Sampling can be used to efficiently implement a basic algorithm for
convex optimization given a separation oracle \textemdash{} the \emph{cutting
plane} method. Convex optimization can be reduced to convex feasibility
by including the objective function as a constraint and doing a binary
search on its value. To solve the feasibility problem, the method
maintains a convex set containing $K$, starting with the ball of
radius $R$ which is guaranteed to contain $K$. At each step it queries
the centroid of the set. If infeasible, it uses the violated inequality
given by the separation oracle to restrict the set. The basis of this
method is the following theorem of Grunbaum.
\begin{thm}[\cite{Grunbaum1960}]
For any convex body $K$, for any halfspace $H$ containing the centroid
of $K$, $\vol(H\cap K)\ge\frac{1}{e}\vol(K)$.
\end{thm}

Thus, the volume of the set maintained decreases by a constant factor
in each iteration and the number of iterations is $O(n\log\frac{R}{r})$,
which is asymptotically the best possible. However, there is one important
difficulty, namely computing the center of gravity, even of an explicit
polytope \cite{DyerF90}, is a computationally intractable problem
(\#P-hard). The next theorem uses sampling to get around this, while
keeping the same asymptotic complexity. 
\begin{thm}[\cite{BV04}]
Let $X=\frac{1}{m}\sum_{i=1}^{m}X_{i}$ where $X_{i}$ are drawn
i.i.d. from a logconcave density $p$ in $\R^{n}$. Then, for any
halfspace $H$ containing $X$, 
\[
\E\left(\int_{H}p\right)\ge\frac{1}{e}-\sqrt{\frac{n}{m}}.
\]
\end{thm}

Thus, for a convex body, using the average of $m=O(n)$ samples is
an effective substitute for the center of gravity. 

While this method achieves the best possible oracle complexity and
(an impractically high) polynomial number of arithmetic operations,
the work of Lee, Sidford and Wong \cite{lee2015faster} shows how
to reduce the overall arithmetic complexity to $\widetilde{O}(n^{3})$.
Their general method also leads to the current fastest algorithms
for semi-definite programming and submodular function minimization. 

Convex optimization given only a membership oracle can also be reduced
to sampling, via the method known as simulated annealing. It starts
with a uniform distribution over the feasible set, then gradually
focuses the distribution on near-optimal points. A canonical way to
minimize the linear function $c^{\top}x$ over a convex body $K$
is to use a sequence of Boltzmann-Gibbs distributions with density
proportional to $e^{-\alpha c^{\top}x}$ for points in $K$, with
$\alpha$ starting close to zero and gradually increasing it. A random
point drawn from this density satisfies

\[
\E\left(c^{\top}x\right)\le\min_{K}c^{\top}x+\frac{n}{\alpha}.
\]
Thus, sampling from the density with $\alpha=n/\epsilon$ gives an
additive $\epsilon$ error approximation. In \cite{KV06} it is shown
how to make this method efficient, using a sequence of only $\widetilde{O}(\sqrt{n})$
distributions. 

\begin{figure}
\begin{tikzpicture}[y=0.80pt, x=0.80pt, yscale=-0.300000, xscale=0.300000, inner sep=0pt, outer sep=0pt]   \path[draw=cff0000,line join=miter,line cap=butt,miter limit=4.00,even odd     rule,line width=1.000pt] (237.0265,445.9110) .. controls (237.0265,445.9110)     and (242.6843,439.7648) .. (255.8418,303.1833) -- (2.5591,198.0155) --     (-220.3297,278.1433) -- (-197.1725,473.4549);   \path[draw=cff0000,line join=miter,line cap=butt,miter limit=4.00,even odd     rule,line width=1.000pt] (-197.1725,473.4549) .. controls (-125.7451,539.7571)     and (22.5800,586.8137) .. (237.0265,445.9110);   \path[draw=c063cb8,line join=miter,line cap=butt,miter limit=4.00,draw     opacity=0.761,even odd rule,line width=1.000pt] (-0.2085,126.8304) .. controls     (36.5871,657.7393) and (32.6447,611.7447) .. (32.6447,611.7447);   \path[draw=black,line join=miter,line cap=butt,miter limit=4.00,even odd     rule,line width=1.000pt] (-56.7162,180.7098) .. controls (-25.1770,585.4621)     and (-25.1770,586.7762) .. (-25.1770,586.7762);   \path[draw=c008000,line join=miter,line cap=butt,miter limit=4.00,even odd     rule,line width=1.000pt] (53.6708,174.1391) .. controls (86.5241,577.5773) and     (86.5241,577.5773) .. (86.5241,577.5773);   \path[draw=black,line join=miter,line cap=butt,miter limit=4.00,even odd     rule,line width=1.000pt] (104.9219,199.1076) .. controls (131.2046,535.5251)     and (131.2046,535.5251) .. (131.2046,535.5251);   \path[draw=c063cb8,line join=miter,line cap=butt,miter limit=4.00,draw     opacity=0.761,even odd rule,line width=1.000pt] (135.1470,209.6283) ..     controls (154.8589,517.1196) and (154.8589,517.1196) .. (154.8589,517.1196);   \path[draw=c008000,line join=miter,line cap=butt,miter limit=4.00,even odd     rule,line width=1.000pt] (158.8013,220.1457) .. controls (179.8274,500.0316)     and (179.8274,500.0316) .. (179.8274,500.0316);   \path[fill=black,line cap=butt,miter limit=4.00,nonzero rule,line width=1.000pt]     (-43.0856,354.2449) ellipse (0.1113cm and 0.1298cm);   \path[fill=black,line cap=butt,miter limit=4.00,nonzero rule,line width=1.000pt]     (116.0222,354.2449) ellipse (0.1113cm and 0.1298cm);   \path[fill=c008000,line cap=butt,miter limit=4.00,nonzero rule,line     width=1.000pt] (69.0490,364.8139) ellipse (0.1113cm and 0.1298cm);   \path[fill=c008000,line cap=butt,miter limit=4.00,nonzero rule,line     width=1.000pt] (168.4896,344.2352) ellipse (0.1113cm and 0.1298cm);   \path[fill=c063cb8,line cap=butt,miter limit=4.00,fill opacity=0.761,nonzero     rule,line width=1.000pt] (12.4574,322.3144) ellipse (0.1113cm and 0.1298cm);   \path[fill=c063cb8,line cap=butt,miter limit=4.00,fill opacity=0.761,nonzero     rule,line width=1.000pt] (142.3187,329.4442) ellipse (0.1113cm and 0.1298cm);   \path[draw=c008000,line join=miter,line cap=butt,miter limit=4.00,even odd     rule,line width=1.000pt] (-240.9693,118.3894) .. controls (298.2213,73.5618)     and (297.6620,72.0101) .. (297.6620,72.0101);   \path[draw=c008000,line join=miter,line cap=butt,miter limit=4.00,even odd     rule,line width=1.000pt] (266.6620,59.0101) .. controls (297.6203,71.0791) and     (293.0616,70.1135) .. (293.0616,70.1135);   \path[draw=c008000,line join=miter,line cap=butt,miter limit=4.00,even odd     rule,line width=1.000pt] (271.6620,84.0101) .. controls (302.6203,71.9411) and     (297.6620,72.0101) .. (297.6620,72.0101);   \path[draw=cff0000,line join=miter,line cap=butt,miter limit=4.00,even odd     rule,line width=1.000pt] (972.7993,434.7532) .. controls (972.7993,434.7532)     and (978.4570,428.6071) .. (991.6146,292.0255) -- (738.3318,186.8578) --     (515.4430,266.9856) -- (538.6003,462.2972);   \path[draw=cff0000,line join=miter,line cap=butt,miter limit=4.00,even odd     rule,line width=1.000pt] (538.6003,462.2972) .. controls (610.0276,528.5994)     and (758.3528,575.6560) .. (972.7993,434.7532);   \path[fill=black,line cap=butt,miter limit=4.00,nonzero rule,line width=1.000pt]     (692.6871,343.0872) ellipse (0.1113cm and 0.1298cm);   \path[fill=black,line cap=butt,miter limit=4.00,nonzero rule,line width=1.000pt]     (815.1752,454.3549) ellipse (0.1113cm and 0.1298cm);   \path[fill=c008000,line cap=butt,miter limit=4.00,nonzero rule,line     width=1.000pt] (804.8217,353.6562) ellipse (0.1113cm and 0.1298cm);   \path[fill=c008000,line cap=butt,miter limit=4.00,nonzero rule,line     width=1.000pt] (904.2622,318.9930) ellipse (0.1113cm and 0.1298cm);   \path[fill=c063cb8,line cap=butt,miter limit=4.00,fill opacity=0.761,nonzero     rule,line width=1.000pt] (925.6949,288.6214) ellipse (0.1113cm and 0.1298cm);   \path[fill=c063cb8,line cap=butt,miter limit=4.00,fill opacity=0.761,nonzero     rule,line width=1.000pt] (878.0914,318.2865) ellipse (0.1113cm and 0.1298cm);   \path[fill=black,line cap=butt,miter limit=4.00,nonzero rule,line width=1.000pt]     (602.5463,296.6083) ellipse (0.1113cm and 0.1298cm);   \path[fill=black,line cap=butt,miter limit=4.00,nonzero rule,line width=1.000pt]     (761.6540,296.6083) ellipse (0.1113cm and 0.1298cm);   \path[fill=c008000,line cap=butt,miter limit=4.00,nonzero rule,line     width=1.000pt] (949.3326,349.9789) ellipse (0.1113cm and 0.1298cm);   \path[fill=c063cb8,line cap=butt,miter limit=4.00,fill opacity=0.761,nonzero     rule,line width=1.000pt] (693.3007,261.8608) ellipse (0.1113cm and 0.1298cm);   \path[fill=c063cb8,line cap=butt,miter limit=4.00,fill opacity=0.761,nonzero     rule,line width=1.000pt] (771.0492,256.3146) ellipse (0.1113cm and 0.1298cm);   \path[fill=black,line cap=butt,miter limit=4.00,nonzero rule,line width=1.000pt]     (722.2646,419.1436) ellipse (0.1113cm and 0.1298cm);   \path[fill=black,line cap=butt,miter limit=4.00,nonzero rule,line width=1.000pt]     (881.3723,419.1436) ellipse (0.1113cm and 0.1298cm);   \path[fill=c008000,line cap=butt,miter limit=4.00,nonzero rule,line     width=1.000pt] (933.8397,409.1338) ellipse (0.1113cm and 0.1298cm);   \path[fill=c063cb8,line cap=butt,miter limit=4.00,fill opacity=0.761,nonzero     rule,line width=1.000pt] (777.8076,387.2130) ellipse (0.1113cm and 0.1298cm);   \path[fill=c063cb8,line cap=butt,miter limit=4.00,fill opacity=0.761,nonzero     rule,line width=1.000pt] (906.2605,283.0752) ellipse (0.1113cm and 0.1298cm);   \path[fill=black,line cap=butt,miter limit=4.00,nonzero rule,line width=1.000pt]     (699.7295,492.3830) ellipse (0.1113cm and 0.1298cm);   \path[fill=black,line cap=butt,miter limit=4.00,nonzero rule,line width=1.000pt]     (916.5836,329.0027) ellipse (0.1113cm and 0.1298cm);   \path[fill=c008000,line cap=butt,miter limit=4.00,nonzero rule,line     width=1.000pt] (960.6002,296.4577) ellipse (0.1113cm and 0.1298cm);   \path[fill=c063cb8,line cap=butt,miter limit=4.00,fill opacity=0.761,nonzero     rule,line width=1.000pt] (798.9343,290.0299) ellipse (0.1113cm and 0.1298cm);   \path[fill=c063cb8,line cap=butt,miter limit=4.00,fill opacity=0.761,nonzero     rule,line width=1.000pt] (942.8801,304.2020) ellipse (0.1113cm and 0.1298cm);   \path[fill=black,line cap=butt,miter limit=4.00,nonzero rule,line width=1.000pt]     (879.9639,365.6224) ellipse (0.1113cm and 0.1298cm);   \path[fill=c063cb8,line cap=butt,miter limit=4.00,fill opacity=0.761,nonzero     rule,line width=1.000pt] (941.1879,333.6919) ellipse (0.1113cm and 0.1298cm);   \path[fill=c008000,line cap=butt,miter limit=4.00,nonzero rule,line     width=1.000pt] (885.9524,268.2887) ellipse (0.1113cm and 0.1298cm);   \path[fill=black,line cap=butt,miter limit=4.00,nonzero rule,line width=1.000pt]     (854.6118,354.3549) ellipse (0.1113cm and 0.1298cm);   \path[fill=c008000,line cap=butt,miter limit=4.00,nonzero rule,line     width=1.000pt] (928.2059,354.2042) ellipse (0.1113cm and 0.1298cm);   \path[fill=c063cb8,line cap=butt,miter limit=4.00,fill opacity=0.761,nonzero     rule,line width=1.000pt] (949.6386,323.8327) ellipse (0.1113cm and 0.1298cm);   \path[fill=c063cb8,line cap=butt,miter limit=4.00,fill opacity=0.761,nonzero     rule,line width=1.000pt] (902.0350,353.4978) ellipse (0.1113cm and 0.1298cm);   \path[fill=c008000,line cap=butt,miter limit=4.00,nonzero rule,line     width=1.000pt] (890.1777,289.4155) ellipse (0.1113cm and 0.1298cm);   \path[fill=c063cb8,line cap=butt,miter limit=4.00,fill opacity=0.761,nonzero     rule,line width=1.000pt] (930.2041,318.2865) ellipse (0.1113cm and 0.1298cm);   \path[fill=black,line cap=butt,miter limit=4.00,nonzero rule,line width=1.000pt]     (972.9216,314.9182) ellipse (0.1113cm and 0.1298cm);   \path[fill=c063cb8,line cap=butt,miter limit=4.00,fill opacity=0.761,nonzero     rule,line width=1.000pt] (979.2161,350.5933) ellipse (0.1113cm and 0.1298cm);   \path[fill=c063cb8,line cap=butt,miter limit=4.00,fill opacity=0.761,nonzero     rule,line width=1.000pt] (966.8238,339.4133) ellipse (0.1113cm and 0.1298cm);   \path[fill=c063cb8,line cap=butt,miter limit=4.00,fill opacity=0.761,nonzero     rule,line width=1.000pt] (967.9485,363.2693) ellipse (0.1113cm and 0.1298cm);   \path[fill=c008000,line cap=butt,miter limit=4.00,nonzero rule,line     width=1.000pt] (909.8960,303.5000) ellipse (0.1113cm and 0.1298cm);   \path[fill=black,line cap=butt,miter limit=4.00,nonzero rule,line width=1.000pt]     (737.7576,268.4393) ellipse (0.1113cm and 0.1298cm);   \path[fill=c008000,line cap=butt,miter limit=4.00,nonzero rule,line     width=1.000pt] (849.8921,279.0083) ellipse (0.1113cm and 0.1298cm);   \path[fill=black,line cap=butt,miter limit=4.00,nonzero rule,line width=1.000pt]     (767.3351,344.4957) ellipse (0.1113cm and 0.1298cm);   \path[fill=black,line cap=butt,miter limit=4.00,nonzero rule,line width=1.000pt]     (926.4428,344.4957) ellipse (0.1113cm and 0.1298cm);   \path[fill=c063cb8,line cap=butt,miter limit=4.00,fill opacity=0.761,nonzero     rule,line width=1.000pt] (822.8781,312.5651) ellipse (0.1113cm and 0.1298cm);   \path[fill=black,line cap=butt,miter limit=4.00,nonzero rule,line width=1.000pt]     (925.0343,290.9746) ellipse (0.1113cm and 0.1298cm);   \path[fill=c063cb8,line cap=butt,miter limit=4.00,fill opacity=0.761,nonzero     rule,line width=1.000pt] (910.2020,268.9031) ellipse (0.1113cm and 0.1298cm);   \path[fill=black,line cap=butt,miter limit=4.00,nonzero rule,line width=1.000pt]     (899.6821,279.7069) ellipse (0.1113cm and 0.1298cm);   \path[fill=c063cb8,line cap=butt,miter limit=4.00,fill opacity=0.761,nonzero     rule,line width=1.000pt] (867.9485,250.5933) ellipse (0.1113cm and 0.1298cm);   \path[fill=black,line cap=butt,miter limit=4.00,nonzero rule,line width=1.000pt]     (795.5040,417.7351) ellipse (0.1113cm and 0.1298cm);   \path[fill=black,line cap=butt,miter limit=4.00,nonzero rule,line width=1.000pt]     (954.6118,417.7351) ellipse (0.1113cm and 0.1298cm);   \path[fill=c063cb8,line cap=butt,miter limit=4.00,fill opacity=0.761,nonzero     rule,line width=1.000pt] (851.0471,385.8046) ellipse (0.1113cm and 0.1298cm);   \path[fill=black,line cap=butt,miter limit=4.00,nonzero rule,line width=1.000pt]     (953.2034,364.2140) ellipse (0.1113cm and 0.1298cm);
\end{tikzpicture}
\centering{}\caption{Centroid cuts vs Simulated annealing for optimization}
\end{figure}
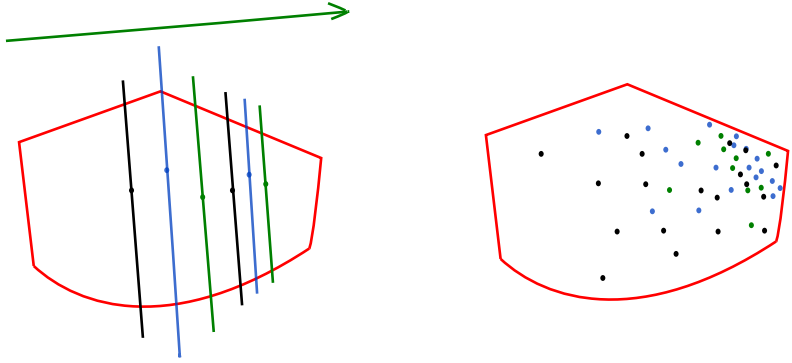

The method can be viewed as a special case of a more general family
of algorithms referred to as the interior-point method \cite{NN1984}.
The latter method works by minimizing the sum of the desired objective
function with a diminishing multiple of a smooth convex function which
blows up at the boundary. Thus the optimum points to this modified
objective for any positive value of the multiplier are in the interior
of the convex body. The path followed by the method as a function
of the multiplier is called the central path. There are several interesting
choices of the smooth convex functions, including the logarithmic
barrier, volumetric barrier, universal barrier and entropic barrier.
The last two of these achieve the optimal rate in terms of the dimension
for arbitrary convex bodies. This rate, i.e., number of steps to reduce
the distance to optimality by a constant factor, is $O(\sqrt{n})$,
the same as for simulated annealing. This is not a coincidence \textemdash{}
the path of the centroid as a function of $\alpha$ corresponds exactly
to the central path taken by the interior-point method using the entropic
barrier \cite{bubeck2014entropic,abernethy2016faster}. The universal
and entropic barriers are not the most efficient for the important
case of linear programming \textemdash{} each step can be implemented
in polynomial time, but is $n^{4}$ or higher. For a linear program
given by $m$ inequalities, the logarithmic barrier uses $\widetilde{O}(\sqrt{m})$
phases, with each phase requiring the solution of a linear system.
Lee and Sidford have proposed a barrier that takes only $\widetilde{O}(\sqrt{n})$
iterations and gives the currest fastest method for solving linear
programs \cite{lee2014path,lee2015efficient}.

Optimization based on sampling has various robustness properties (e.g.,
it can be applied to stochastic optimization \cite{belloni2015escaping,FeldmanGV17}
and regret minimization \cite{narayanan2010random,bubeck2017kernel}),
and continues to be an active research topic. 

\subsubsection{Sampling to Volume computation and Integration}

\begin{figure}
\centering{}\begin{tikzpicture}[y=0.80pt, x=0.80pt, yscale=-0.300000, xscale=0.300000, inner sep=0pt, outer sep=0pt]   \path[draw=cff0000,line join=miter,line cap=butt,miter limit=4.00,even odd     rule,line width=1.000pt] (293.0305,204.8597) .. controls (293.0305,204.8597)     and (293.0305,204.8597) .. (121.8612,67.6541) .. controls (121.8612,67.6541)     and (124.5357,67.6541) .. (-220.4774,199.5826) .. controls     (-188.3831,389.5597) and (-188.3831,389.5597) .. (-188.3831,389.5597);   \path[draw=cff0000,line join=miter,line cap=butt,miter limit=4.00,even odd     rule,line width=1.000pt] (-188.3831,389.5597) .. controls (-8.0223,511.0899)     and (148.2237,420.0602) .. (293.0305,204.8597);   \path[draw=c063cb8,line cap=butt,miter limit=4.00,draw opacity=0.761,fill     opacity=0.761,nonzero rule,line width=1.000pt] (17.5369,281.4858) ellipse     (1.4757cm and 1.3246cm);   \path[draw=c063cb8,dash pattern=on 5.36pt off 5.36pt,line cap=butt,miter     limit=4.00,draw opacity=0.761,fill opacity=0.761,nonzero rule,line     width=1.000pt] (16.3877,278.8490) ellipse (2.4811cm and 2.2623cm);   \path[draw=c063cb8,dash pattern=on 5.36pt off 5.36pt,line cap=butt,miter     limit=4.00,draw opacity=0.761,fill opacity=0.761,nonzero rule,line     width=1.000pt] (17.5370,278.3215) ellipse (3.5514cm and 3.2595cm);   \path[draw=c063cb8,dash pattern=on 5.36pt off 5.36pt,line cap=butt,miter     limit=4.00,draw opacity=0.761,fill opacity=0.761,nonzero rule,line     width=1.000pt] (17.5369,279.3763) ellipse (4.8811cm and 4.3609cm);   \begin{scope}[cm={{0.79837,0.24132,-0.54061,0.35637,(702.59085,103.82937)}},draw=c063cb8,draw opacity=0.761,transparency group]     \path[draw=c063cb8,line join=miter,line cap=butt,miter limit=4.00,draw       opacity=0.761,even odd rule,line width=1.000pt] (677.9334,367.9549) ..       controls (677.9334,367.9549) and (677.9334,367.9549) .. (481.9488,196.7899) ..       controls (481.9488,196.7899) and (485.0110,196.7899) .. (89.9794,361.3717) ..       controls (126.7265,598.3694) and (126.7265,598.3694) .. (126.7265,598.3694);     \path[draw=c063cb8,line join=miter,line cap=butt,miter limit=4.00,draw       opacity=0.761,even odd rule,line width=1.000pt] (126.7265,598.3694) ..       controls (333.2353,749.9792) and (512.1331,636.4190) .. (677.9334,367.9549);     \path[draw=c008000,line cap=butt,miter limit=4.00,fill opacity=0.761,nonzero       rule,line width=1.000pt] (362.5000,463.5464) ellipse (1.6896cm and 1.6525cm);   \end{scope}   \path[draw=c008000,line join=miter,line cap=butt,miter limit=4.00,even odd     rule,line width=1.000pt] (567.7194,351.6717) .. controls (691.9218,393.7032)     and (723.3844,-201.4376) .. (844.1567,318.0027) .. controls     (865.8678,397.4897) and (929.7863,384.7063) .. (929.7863,384.7063);   \path[draw=c063cb8,dash pattern=on 6.40pt off 6.40pt,line join=miter,line     cap=butt,miter limit=4.00,draw opacity=0.761,even odd rule,line width=1.000pt]     (532.0563,336.7374) .. controls (793.1329,25.0108) and (802.9621,228.2421) ..     (968.7946,337.0959);
\end{tikzpicture}\caption{DFK vs Simulated Annealing}
\end{figure}
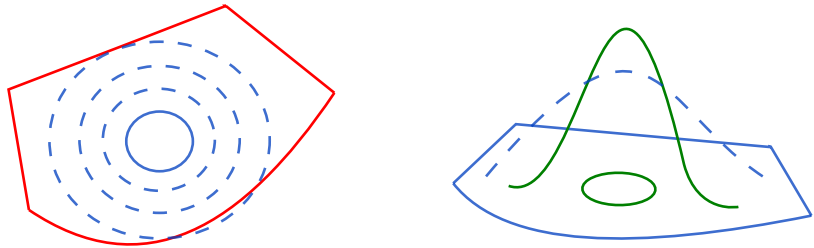

Sampling is the core of efficient volume computation and integration.
The main idea for the latter problems is to sample a sequence of logconcave
distributions, starting with one that is easy to integrate and ending
with the function whose integral is desired. This process, known as
\emph{simulated annealing} can be expressed as the following telescoping
product:

\[
\int_{\R^{n}}f=\int f_{0}\frac{\int f_{i}}{\int f_{0}}\frac{\int f_{2}}{\int f_{3}}\ldots\frac{\int f_{m}}{\int f_{m-1}}
\]
where $f_{m}=f$. Each ratio $\int f_{i+1}/\int f_{i}$ is the expectation
of the estimator $Y=\frac{f_{i+1}(X)}{f_{i}(X)}$ for $X$ drawn from
the density proportional to $f_{i}$. What is the optimal sequence
of interpolating functions to use? The celebrated polynomial-time
algorithm of Dyer, Frieze and Kannan \cite{DyerFK89} used the uniform
distribution on a sequence of convex bodies, starting with the ball
contained inside the input body $K$. Each body in the sequence is
a ball intersected with the given convex body $K$: $K_{i}=2^{\frac{i}{n}}rB\cap K$.
The length of the sequence is $m=O(n\log\frac{R}{r})$ so that the
final body is just $K$. A variance bound shows that $O(m/\epsilon^{2}$)
samples per distribution suffice to get an overall $1+\epsilon$ multiplicative
error approximation with high probability. The total number of samples
is $O^{*}(m^{2})=O^{*}(n^{2})$ and the complexity of the resulting
algorithm is $O^{*}(n^{5})$ as shown in \cite{KLS97}. Table \ref{tab:volume}
below summarizes progress on the volume problem over the past three
decades. Besides improving the complexity of volume computation, each
step has typically resulted in new techniques. For more details, we
refer the reader to surveys on the topic \cite{Simonovits03,VemSurvey}.

\begin{table}[h]
\centering{}%
\begin{tabular}{|c|l|c|}
\hline 
Year/Authors  & New ingredients & Steps\tabularnewline
\hline 
\hline 
1989/Dyer-Frieze-Kannan \cite{DyerFK89} & Everything & $n^{23}$\tabularnewline
\hline 
1990/Lovász-Simonovits \cite{LS90} & Better isoperimetry & $n^{16}$\tabularnewline
\hline 
1990/Lovász \cite{L90} & Ball walk & $n^{10}$\tabularnewline
\hline 
1991/Applegate-Kannan \cite{ApplegateK91} & Logconcave sampling & $n^{10}$\tabularnewline
\hline 
1990/Dyer-Frieze \cite{DyerF90} & Better error analysis & $n^{8}$\tabularnewline
\hline 
1993/Lovász-Simonovits \cite{LS93} & Localization lemma & $n^{7}$\tabularnewline
\hline 
1997/Kannan-Lovász-Simonovits \cite{KLS97} & Speedy walk, isotropy & $n^{5}$\tabularnewline
\hline 
2003/Lovász-Vempala \cite{LV2} & Annealing, hit-and-run & $n^{4}$\tabularnewline
\hline 
2015/Cousins-Vempala \cite{CV2015} (well-rounded) & Gaussian Cooling & $n^{3}$\tabularnewline
\hline 
2017/Lee-Vempala (polytopes) & Hamiltonian Walk & $mn^{\frac{2}{3}}$\tabularnewline
\hline 
\end{tabular}\caption{\label{tab:volume} The complexity of volume estimation, each step
uses $\widetilde{O}(n)$ bit of randomness. }
\end{table}

Lovász and Vempala \cite{LV2} improved on \cite{KLS97} by sampling
from a sequence of nonuniform distributions. The densities in the
sequence have the form $f_{i}(x)\propto\exp(-\alpha_{i}\norm x)\chi_{K}(x)$
or $f_{i}(x)\propto\exp(-\alpha_{i}\norm x^{2})\chi_{K}(x)$. Then
the ratio of two consecutive integrals is the expectation of the following
estimator:

\[
Y=\frac{f_{i+1}(X)}{f_{i}(X)}.
\]
They showed that the coefficient $\alpha_{i}$ (inverse ``temperature'')
can be changed by a factor of $(1+\frac{1}{\sqrt{n}})$, which implies
that $m=\widetilde{O}(\sqrt{n})$ phases suffice, and the total number
of samples is only $O^{*}(n)$. This is perhaps surprising since the
ratio of the initial integral to the final is typically $n^{\Omega(n)}$.
Even though the algorithm uses only $\widetilde{O}(\sqrt{n})$ phases,
and hence estimates a ratio of $n^{\tilde{\Omega}(\sqrt{n})}$ in
one or more phases, the variance of the estimator is bounded in every
phase. 
\begin{thm}[\cite{LV2}]
The volume of a convex body in $\R^{n}$ (given by a membership oracle)
can be computed to relative error $\varepsilon$ using $\widetilde{O}(n^{4}/\varepsilon^{2})$
oracle queries and $\widetilde{O}(n^{2})$ arithmetic operations per
query.
\end{thm}

The LV algorithm has two parts. In the first it finds a transformation
that puts the body in near-isotropic position. The complexity of this
part is $\widetilde{O}(n^{4})$. In the second part, it runs the annealing
schedule, while maintaining that the distribution being sampled is
\emph{well-rounded}, a weaker condition than isotropy. Well-roundedness
requires that a level set of measure $\frac{1}{8}$ contains a constant-radius
ball and the trace of the covariance (expected squared distance of
a random point from the mean) to be bounded by $O(n)$, so that $R/r$
is effectively $O(\sqrt{n})$. To achieve the complexity guarantee
for the second phase, it suffices to use the KLS bound of $\psi_{p}\gtrsim n^{-\frac{1}{2}}$.
Connecting improvements in the Cheeger constant directly to the complexity
of volume computation is an open question. To apply improvements in
the Cheeger constant, one would need to replace well-roundedness with
(near-)isotropy and maintain that. However, maintaining isotropy appears
to be much harder \textemdash{} possibly requiring a sequence of $\Omega(n)$
distributions and $\Omega(n$) samples from each, providing no gain
over the current complexity of $O^{*}(n^{4})$ even if the KLS conjecture
turns out to be true.

Cousins and Vempala \cite{CV2015} gave a faster algorithm for well-rounded
convex bodies (any isotropic logconcave density satisfies $\frac{R}{r}=O(\sqrt{n})$
and is well-rounded). Their algorithm, called Gaussian cooling, is
significantly simpler, crucially utilizes the fact that the KLS conjecture
holds for a Gaussian density restricted by any convex body (Theorem
\ref{thm:Gaussian-iso}), and completely avoids computing an isotropic
transformation.
\begin{thm}[\cite{CV2015}]
The volume of a well-rounded convex body, i.e., with $R/r=O^{*}(\sqrt{n})$,
can be computed using $O^{*}(n^{3})$ oracle calls.
\end{thm}

We note that the covariance matrix of any logconcave density can be
computed efficiently from only a linear in the dimension number of
samples. This question of sample complexity was also motivated by
the study of volume computation.
\begin{thm}[\cite{Adamczak2010,adamczak2011sharp,srivastava2013covariance}]
Let $Y=\frac{1}{m}\sum_{i=1}^{m}X_{i}X_{i}^{\top}$ where $X_{1},\ldots,X_{m}$
are drawn from an isotropic logconcave density $p$. If $m\gtrsim\frac{n}{\varepsilon^{2}}$,
then $\norm{Y-I}_{\op}\le\varepsilon$ with high probability.
\end{thm}

\section{Proof techniques}

Classical proofs of isoperimetry for special distributions are based
on different types of symmetrization that effectively identify the
extremal subsets. Bounding the Cheeger constant for general convex
bodies and logconcave densities is more complicated since the extremal
sets can be nonlinear and hard to describe precisely, due to the trade-off
between minimizing the boundary measure of a subset and utilizing
as much of the ``external'' boundary as possible. The main technique
to prove bounds in the general setting has been \emph{localization},
a method to reduce inequalities in high dimension to inequalities
in one dimension. We now describe this technique with a few applications. 

\subsection{Localization}

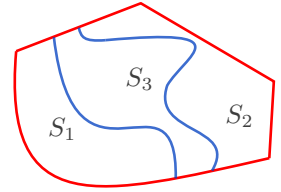
\begin{wrapfigure}{r}{0.25\textwidth}
\centering
\begin{tikzpicture}[y=0.80pt, x=0.80pt, yscale=-0.200000, xscale=0.200000, inner sep=0pt, outer sep=0pt]   \path[draw=cff0000,line join=miter,line cap=butt,miter limit=4.00,even odd     rule,line width=1.000pt] (84.8522,415.4839) .. controls (377.3462,302.3555)     and (380.0296,302.3555) .. (380.0296,302.3555) -- (490.3539,367.2152) --     (693.9910,486.9334) -- (683.2572,668.5343);   \path[draw=cff0000,line join=miter,line cap=butt,miter limit=4.00,even odd     rule,line width=1.000pt] (84.8522,415.4839) .. controls (50.3252,819.2758) and     (284.0089,759.4034) .. (683.2572,668.5343);   \path[draw=c063cb8,line join=miter,line cap=butt,miter limit=4.00,draw     opacity=0.761,even odd rule,line width=1.000pt] (174.0000,382.3622) ..     controls (223.1713,796.0781) and (466.8943,432.1223) .. (462.0000,716.3622);   \path[draw=c063cb8,line join=miter,line cap=butt,miter limit=4.00,draw     opacity=0.761,even odd rule,line width=1.000pt] (296.0000,389.3622) ..     controls (625.9415,359.4824) and (483.4065,426.7886) .. (442.0000,496.3622) ..     controls (402.3261,614.9930) and (611.7898,588.9575) .. (548.0000,698.1069);   \path[draw=c063cb8,line join=miter,line cap=butt,miter limit=4.00,draw     opacity=0.761,even odd rule,line width=1.000pt] (232.9787,359.8090) ..     controls (238.3701,382.0707) and (243.5154,393.1383) .. (296.0000,389.3622);   \path[draw=c063cb8,fill=c063cb8,line join=miter,line cap=butt,miter     limit=4.00,draw opacity=0.761,fill opacity=0.761,line width=1.000pt]     (160.5745,624.7026) node[above right] (text6789) {$S_1$};   \path[draw=c063cb8,fill=c063cb8,line join=miter,line cap=butt,miter     limit=4.00,draw opacity=0.761,fill opacity=0.761,line width=1.000pt]     (343.6170,506.6175) node[above right] (text6793) {$S_3$};   \path[draw=c063cb8,fill=c063cb8,line join=miter,line cap=butt,miter     limit=4.00,draw opacity=0.761,fill opacity=0.761,line width=1.000pt]     (579.7872,593.8516) node[above right] (text6797) {$S_2$};
\end{tikzpicture}
\caption*{Euclidean isoperimetry}
\end{wrapfigure}

We will sketch a proof of the following theorem to illustrate the
use of localization. This theorem was also proved by Karzanov and
Khachiyan \cite{KarzanovK91} using a different, more direct approach.
\begin{thm}[\cite{DyerF90,LS90,KarzanovK91}]
\label{ISO2}Let $f$ be a logconcave function whose support has
diameter $D$ and let $\pi_{f}$ be the induced measure. Then for
any partition of $\R^{n}$ into measurable sets $S_{1},S_{2},S_{3}$, 

\[
\pi_{f}(S_{3})\geq\frac{2d(S_{1},S_{2})}{D}\min\{\pi_{f}(S_{1}),\pi_{f}(S_{2})\}.
\]
\end{thm}

Before discussing the proof, we note that there is a variant of this
result in the Riemannian setting.
\begin{thm}[\cite{li1980estimates}]
If $K\subset(M,g)$ is a locally convex bounded domain with smooth
boundary, diameter $D$ and $Ric_{g}\geq0$, then the Poincaré constant
is at least $\frac{\pi^{2}}{4D^{2}}$, i.e., for any $g$ with $\int g=0$,
we have that
\[
\int\left|\nabla g(x)\right|^{2}dx\geq\frac{\pi^{2}}{4D^{2}}\int g(x)^{2}dx.
\]
\end{thm}

For the case of convex bodies in $\Rn$, this result is equivalent
to Theorem \ref{ISO2} up to a constant. One benefit of localization
is that it does not require a carefully crafted potential. Localization
has recently been generalized to Riemannian setting \cite{klartag2017needle}.
The origins of this method were in a paper by Payne and Weinberger
\cite{payne1960optimal}. 

We begin the proof of Theorem \ref{ISO2}. For a proof by contradiction,
let us assume the converse of its conclusion, i.e., for some partition
$S_{1},S_{2},S_{3}$ of $\R^{n}$ and logconcave density $f$, assume
that

\[
\int_{S_{3}}f(x)\,dx<C\int_{S_{1}}f(x)\,dx\quad\mbox{ and }\quad\int_{S_{3}}f(x)\,dx<C\int_{S_{2}}f(x)\,dx
\]
where $C=2d(S_{1},S_{2})/D$. This can be reformulated as 

\begin{equation}
\int_{\R^{n}}g(x)\,dx>0\quad\mbox{ and }\quad\int_{\R^{n}}h(x)\,dx>0\label{eq:1}
\end{equation}
where 
\[
g(x)=\begin{cases}
Cf(x) & \mbox{ if }x\in S_{1},\\
0 & \mbox{ if }x\in S_{2},\\
-f(x) & \mbox{ if }x\in S_{3}.
\end{cases}\quad\mbox{ and }\quad h(x)=\begin{cases}
0 & \mbox{ if }x\in S_{1},\\
Cf(x) & \mbox{ if }x\in S_{2},\\
-f(x) & \mbox{ if }x\in S_{3}.
\end{cases}
\]
 These inequalities are for functions in $\R^{n}$. The next lemma
will help us analyze them.
\begin{lem}[Localization Lemma \cite{KLS95}]
\label{lem:LOCALIZATION}Let $g,h:\R^{n}\rightarrow\R$ be lower
semi-continuous integrable functions such that 
\[
\int_{\R^{n}}g(x)\,dx>0\quad\mbox{ and }\quad\int_{\R^{n}}h(x)\,dx>0.
\]
 Then there exist two points $a,b\in\R^{n}$ and an affine function
$\ell:[0,1]\rightarrow\R_{+}$ such that 
\[
\int_{0}^{1}\ell(t)^{n-1}g((1-t)a+tb)\,dt>0\quad\mbox{ and }\quad\int_{0}^{1}\ell(t)^{n-1}h((1-t)a+tb)\,dt>0.
\]
\end{lem}

\begin{wrapfigure}{r}{0.25\textwidth}
\centering
\begin{tikzpicture}[y=0.80pt, x=0.80pt, yscale=-0.250000, xscale=0.250000, inner sep=0pt, outer sep=0pt]   \path[draw=black,fill=cff6666,line cap=butt,miter limit=4.00,fill     opacity=0.761,nonzero rule,line width=1.000pt] (120.7860,427.2994) ellipse     (0.4222cm and 1.0081cm);   \path[draw=black,fill=cff6666,line cap=butt,miter limit=4.00,fill     opacity=0.758,nonzero rule,line width=1.0pt] (564.1912,414.3622) ellipse     (1.3706cm and 2.5744cm);   \path[draw=black,line join=miter,line cap=butt,miter limit=4.00,even odd     rule,line width=1pt] (119.7851,425.6119) .. controls (564.1911,416.6121)     and (562.1893,416.6121) .. (562.1893,416.6121);   \path[draw=black,line join=miter,line cap=butt,miter limit=4.00,even odd     rule,line width=1pt] (119.5685,391.8627) .. controls (563.9745,324.3643)     and (560.1875,323.2393) .. (560.1875,323.2393);   \path[draw=black,line join=miter,line cap=butt,miter limit=4.00,even odd     rule,line width=1pt] (120.0068,462.8607) -- (561.7147,505.1406);
\end{tikzpicture}
\caption*{Truncated cone}
\end{wrapfigure}
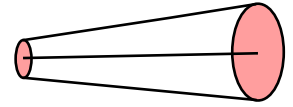

The points $a,b$ represent an interval and one may think of $\ell(t)^{n-1}$
as proportional to the cross-sectional area of an infinitesimal cone.

The lemma says that over this cone truncated at $a$ and $b$, the
integrals of $g$ and $h$ are positive. Also, without loss of generality,
we can assume that $a,b$ are in the union of the supports of $g$
and $h$. 
\begin{proof}[Proof outline]
The main idea is the following. Let $H$ be any halfspace such that
\[
\int_{H}g(x)\,dx=\frac{1}{2}\int_{\R^{n}}g(x)\,dx.
\]
Let us call this a bisecting halfspace. Now either 
\[
\int_{H}h(x)\,dx>0\quad\mbox{ or }\quad\int_{\R^{n}\setminus H}h(x)\,dx>0.
\]
Thus, either $H$ or its complementary halfspace will have positive
integrals for both $g$ and $h$, reducing the domain of the integrals
from $\R^{n}$ to a halfspace. If we could repeat this, we might hope
to reduce the dimensionality of the domain. For any $(n-2)$-dimensional
affine subspace $L$, there is a bisecting halfspace containing $L$
in its bounding hyperplane. To see this, let $H$ be a halfspace containing
$L$ in its boundary. Rotating $H$ about $L$ we get a family of
halfspaces with the same property. This family includes $H'$, the
complementary halfspace of $H$. The function $\int_{H}g-\int_{\R^{n}\setminus H}g$
switches sign from $H$ to $H'$. Since this is a continuous family,
there must be a halfspace for which the function is zero.

If we take all $(n-2)$-dimensional affine subspaces defined by $\left\{ x\in\Rn:\ x_{i}=r_{1},\,x_{j}=r_{2}\right\} $
where $r_{1},r_{2}$ are rational, then the intersection of all the
corresponding bisecting halfspaces is a line or a point (by choosing
only rational values for $x_{i}$, we are considering a countable
intersection). To see why it is a line or a point, assume we are left
with a two or higher dimensional set. Since the intersection is convex,
there is a point in its interior with at least two coordinates that
are rational, say $x_{1}=r_{1}$ and $x_{2}=r_{2}$. But then there
is a bisecting halfspace $H$ that contains the affine subspace given
by $x_{1}=r_{1},\,x_{2}=r_{2}$ in its boundary, and so it properly
partitions the current set.

Thus the limit of this bisection process is a function supported on
an interval (which could be a single point), and since the function
itself is a limit of convex sets (intersections of halfspaces) containing
this interval, it is a limit of a sequence of concave functions and
is itself concave, with positive integrals. Simplifying further from
concave to linear takes quite a bit of work. For the full proof, we
refer the reader to \cite{LS93}.
\end{proof}
Going back to the proof sketch of Theorem \ref{ISO2}, we can apply
the localization lemma to get an interval $[a,b]$ and an affine function
$\ell$ such that 
\begin{equation}
\int_{0}^{1}\ell(t)^{n-1}g((1-t)a+tb)\,dt>0\quad\mbox{ and }\quad\int_{0}^{1}\ell(t)^{n-1}h((1-t)a+tb)\,dt>0.\label{AFTERLEMMA}
\end{equation}
The functions $g,h$ as we have defined them are not lower semi-continuous.
However, this can be addressed by expanding $S_{1}$ and $S_{2}$
slightly so as to make them open sets, and making the support of $f$
an open set. Since we are proving strict inequalities, these modifications
do not affect the conclusion.

Let us partition $[0,1]$ into $Z_{1},Z_{2},Z_{3}$ as follows: 
\[
Z_{i}=\{t\in[0,1]\,:\,(1-t)a+tb\in S_{i}\}.
\]
Note that for any pair of points $u\in Z_{1},v\in Z_{2},\,|u-v|\ge d(S_{1},S_{2})/D$.
We can rewrite (\ref{AFTERLEMMA}) as 
\[
\int_{Z_{3}}\ell(t)^{n-1}f((1-t)a+tb)\,dt<C\int_{Z_{1}}\ell(t)^{n-1}f((1-t)a+tb)\,dt
\]
and 
\[
\int_{Z_{3}}\ell(t)^{n-1}f((1-t)a+tb)\,dt<C\int_{Z_{2}}\ell(t)^{n-1}f((1-t)a+tb)\,dt.
\]
The functions $f$ and $\ell(\cdot)^{n-1}$ are both logconcave, so
$F(t)=\ell(t)^{n-1}f((1-t)a+tb)$ is also logconcave. We get, 
\begin{equation}
\int_{Z_{3}}F(t)\,dt<C\min\left\{ \int_{Z_{1}}F(t)\,dt,\int_{Z_{2}}F(t)\,dt\right\} .\label{FEQ}
\end{equation}
Now consider what Theorem \ref{ISO2} asserts for the function $F(t)$
over the interval $[0,1]$ and the partition $Z_{1},Z_{2},Z_{3}$:
\begin{equation}
\int_{Z_{3}}F(t)\,dt\geq2d(Z_{1},Z_{2})\min\left\{ \int_{Z_{1}}F(t)\,dt,\int_{Z_{2}}F(t)\,dt\right\} .\label{FEQ2}
\end{equation}
We have substituted $1$ for the diameter of the interval $[0,1]$.
Also, $2d(Z_{1},Z_{2})\geq2d(S_{1},S_{2})/D=C$. Thus, Theorem \ref{ISO2}
applied to the function $F(t)$ contradicts (\ref{FEQ}) and to prove
the theorem in general, and it suffices to prove it in the one-dimensional
case. A combinatorial argument reduces this to the case when each
$Z_{i}$ is a single interval. Proving the resulting inequality up
to a factor of $2$ is a simple exercise and uses only the unimodality
of $F$. The improvement to the tight bound requires one-dimensional
logconcavity. This completes the proof of Theorem \ref{ISO2}.

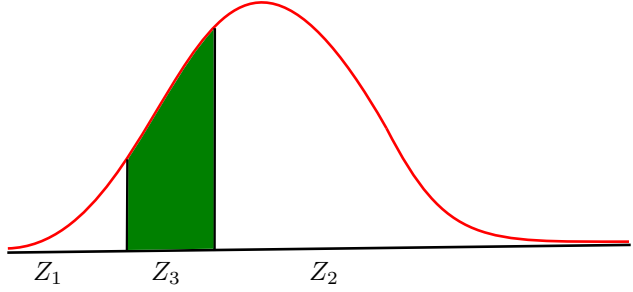
\begin{figure}
\centering{}\vspace*{-3cm}
\begin{tikzpicture}[y=0.80pt, x=0.80pt, yscale=-0.500000, xscale=0.500000, inner sep=0pt, outer sep=0pt]   \path[draw=black,line join=miter,line cap=butt,miter limit=4.00,even odd     rule,line width=1.000pt] (49.4368,481.0238) -- (638.9274,473.7382);   \path[draw=cff0000,line join=miter,line cap=butt,miter limit=4.00,even odd     rule,line width=1.000pt] (49.9408,477.0956) .. controls (211.8739,472.0423)     and (225.1589,42.0372) .. (408.0000,363.3622);   \path[draw=cff0000,line join=miter,line cap=butt,miter limit=4.00,even odd     rule,line width=1.000pt] (408.0000,363.3622) .. controls (465.1394,474.3529)     and (497.3013,470.6621) .. (637.3252,470.6447);   \path[draw=black,line join=miter,line cap=butt,miter limit=4.00,even odd     rule,line width=1.000pt] (163.1089,392.7482) .. controls (162.6487,480.1715)     and (162.6487,480.1715) .. (162.6487,480.1715);   \path[draw=black,line join=miter,line cap=butt,miter limit=4.00,even odd     rule,line width=1.000pt] (245.6003,268.6231) .. controls (245.3088,479.5246)     and (245.3088,479.5246) .. (245.3088,479.5246);   \path[xscale=0.943,yscale=1.060,fill=c008600,line join=miter,line cap=butt,line     width=1.000pt] (196.7728,481.5306) node[above right] (text7961) {$Z_3$};   \path[xscale=0.943,yscale=1.060,fill=black,line join=miter,line cap=butt,line     width=1.000pt] (355.1962,481.5861) node[above right] (text7961-8) {$Z_2$};   \path[xscale=0.943,yscale=1.060,fill=black,line join=miter,line cap=butt,line     width=1.000pt] (78.3270,482.0243) node[above right] (text7961-7) {$Z_1$};   \path[fill=c008600] (163.8199,435.1567) -- (164.1599,391.7936) --     (168.7722,384.6251) .. controls (174.7102,375.3961) and (178.2118,369.6886) ..     (193.1964,344.8145) .. controls (219.2131,301.6275) and (227.9729,288.7495) ..     (243.0990,271.4510) -- (244.3579,270.0115) -- (244.3768,373.6738) --     (244.3957,477.3362) -- (222.9072,477.6342) .. controls (211.0885,477.7981) and     (192.8825,478.0644) .. (182.4493,478.2260) -- (163.4800,478.5199) --     (163.8200,435.1567) -- cycle;
\end{tikzpicture}\caption{One-dimensional isoperimetry}
\end{figure}

The localization lemma has been used to prove a variety of isoperimetric
inequalities. The next theorem is a refinement of Theorem \ref{ISO2},
replacing the diameter by the square-root of the expected squared
distance of a random point from the mean. For an isotropic distribution
this is an improvement from $n$ to $\sqrt{n}$. This theorem was
proved by Kannan-Lovász-Simonovits in the same paper in which they
proposed the KLS conjecture. 
\begin{thm}[\cite{KLS95}]
\label{thm:KLS_thm} For any logconcave density $p$ in $\R^{n}$
with covariance matrix $A$, the KLS constant satisfies
\[
\psi_{p}\gtrsim\frac{1}{\sqrt{\tr(A)}}.
\]
\end{thm}

The next theorem shows that the KLS conjecture is true for an important
family of distributions. The proof is again by localization \cite{CV2014},
and the one-dimensional inequality obtained is a Brascamp-Lieb Theorem.
We note that the same theorem can be obtained by other means \cite{Ledoux1999,Eldan2013}.
\begin{thm}
\label{thm:Gaussian-iso}Let $h(x)=f(x)e^{-\frac{1}{2}x^{\top}Bx}/\int f(y)e^{-\frac{1}{2}y^{\top}By}dy$
where $f:\R^{n}\rightarrow\R_{+}$ is an integrable logconcave function
and $B$ is positive definite. Then $h$ is logconcave and for any
measurable subset $S$ of $\R^{n}$,
\[
\frac{h(\partial S)}{\min\left\{ h(S),h\left(\R^{n}\setminus S\right)\right\} }\gtrsim\frac{1}{\norm{B^{-1}}_{\op}^{\frac{1}{2}}}.
\]
In other words, the expansion of h is $\Omega(\norm{B^{-1}}_{\op}^{-\frac{1}{2}})$.
\end{thm}

The analysis of the Gaussian Cooling algorithm for volume computation
\cite{CV2015} uses localization. Kannan, Lovász and Montenegro \cite{KannanLM06}
used localization to prove the following bound on the log-Cheeger
constant, a quantity that is asymptotically the square-root of the
log-Sobolev constant \cite{ledoux1994simple}, and which we will discuss
in the next section.
\begin{thm}[\cite{KannanLM06}]
\label{thm:-The-log-Cheeger} The log-Cheeger constant of any isotropic
logconcave density with support of diameter $D$ satisfies $\kappa_{p}\gtrsim\frac{1}{D}$
where
\[
\kappa_{p}=\inf_{S\subseteq\R^{n}}\frac{p(\partial S)}{\min\left\{ p(S)\log\frac{1}{p(S)},p(\R^{n}\setminus S)\log\frac{1}{p(\R^{n}\setminus S)}\right\} }.
\]
\end{thm}

Next we mention an application to the anti-concentration of polynomials.
This is a corollary of a more general result by Carbery and Wright.
\begin{thm}[\cite{carbery2001distributional}]
Let $q$ be a degree $d$ polynomial in $\R^{n}$. Then for a convex
body $K\subset\R^{n}$ of volume $1$, any $\epsilon>0$, and $x$
drawn uniform from $K$, 
\[
\Pr_{x\sim K}\left(|q(x)|\le\epsilon\max_{K}|q(x)|\right)\lesssim\epsilon^{\frac{1}{d}}n
\]
\end{thm}

We conclude this section with a nice interpretation of the localization
lemma by Fradelizi and Guedon. They also give a version that extends
localization to multiple inequalities.
\begin{thm}[Reformulated Localization Lemma \cite{fradelizi2004extreme}]
Let $K$ be a compact convex set in $\Rn$ and $f$ be an upper semi-continuous
function. Let $P_{f}$ be the set of logconcave distributions $\mu$
supported by $K$ satisfying $\int fd\mu\geq0$. The set of extreme
points of $\text{conv}P_{f}$ is exactly:
\begin{enumerate}
\item the Dirac measure at points $x$ such that $f(x)\geq0$, or
\item the distributions $v$ satisfies
\begin{enumerate}
\item density function is of the form $e^{\ell}$ with linear $\ell$,
\item support equals to a segment $[a,b]\subset K$,
\item $\int fdv=0$,
\item $\int_{a}^{x}fdv>0$ for $x\in(a,b)$ or $\int_{x}^{b}fdv>0$ for
$x\in(a,b)$.
\end{enumerate}
\end{enumerate}
\end{thm}

Since we know the maximizer of any convex function is at extreme points,
this shows that one can optimize $\max_{\mu\in P_{f}}\Phi(\mu)$ for
any convex $\Phi$ by checking Dirac measures and log-affine functions!

\subsection{Stochastic localization}

We now describe a variant of localization that is the key idea behind
recent progress on the KLS conjecture. Consider a subset $E$ with
measure $0.5$ according to a logconcave density (it suffices to consider
such subsets to bound the isoperimetric constant \cite{Milman2009}).
In standard localization, we repeatedly bisect space using a hyperplane
that preserves the relative measure of $E$. The limit of this process
is a partition into 1-dimensional logconcave measures (``needles''),
for which inequalities are easier to prove. This approach runs into
difficulties for proving the KLS conjecture. While the original measure
might be isotropic, the one-dimensional measures could, in principle,
have variances roughly equal to the trace of the original covariance
(i.e., long thin needles), for which the Cheeger constant is much
smaller. However, if we pick the bisecting hyperplanes randomly, it
seems unlikely that we get such long thin needles. In a different
line of work, Klartag introduced a \textquotedbl{}tilt\textquotedbl{}
operator, i.e., multiplying a density by an exponential of the form
$f(x)\propto e^{-c^{\top}x}$ along a fixed vector $c$, and used
it in his paper improving the slicing constant \cite{Klartag2006,klartag2012centroid}.
A stochastic version of localization combining both these aspects,
i.e., apprpaching needles via tilt operators, was discovered by Eldan
\cite{Eldan2013}. 

Stochastic localization can be viewed as the continuous-time process,
where at each step, we pick a random direction and multiply the current
density with a linear function along the chosen direction. Thus, the
discrete bisection step is replaced by infinitesimal steps, each of
which is a renormalization with a linear function in a random direction.
One might view this informally as an averaging over infinitesimal
needles. The discrete time equivalent would be $p_{t+1}(x)=p_{t}(x)(1+\sqrt{h}(x-\mu_{t})^{\top}w)$
for a sufficiently small $h$ and random Gaussian vector $w$. Using
the approximation $1+y\sim e^{y-\frac{1}{2}y^{2}}$, we see that over
time this process introduces a negative quadratic factor in the exponent,
which will be the Gaussian factor. As time tends to $\infty$, the
distribution tends to a more and more concentrated Gaussian and eventually
a delta function, at which point any subset has measure either $0$
or 1. The idea of the proof is to stop at a time that is large enough
to have a strong Gaussian factor in the density, but small enough
to ensure that the measure of a set is not changed by more than a
constant. Over time, the density can be viewed as a spherical Gaussian
times a logconcave function, with the Gaussian gradually reducing
in variance. When the Gaussian becomes sufficiently small in variance,
then the overall distribution has good isoperimetric coefficient,
determined by the inverse of the Gaussian standard deviation (Theorem
\ref{thm:Gaussian-iso}). An important property of the infinitesimal
change at each step is \emph{balance \textendash{} }the density at
time $t$ is a martingale and therefore the expected measure of any
subset is the same as the original measure. Over time, the measure
of a set $E$ is a random quantity that deviates from its original
value of $\frac{1}{2}$ over time. The main question then is: what
direction to use at each step so that (a) the measure of $E$ remains
bounded and (b) the Gaussian factor in the density eventually has
small variance.

\begin{figure}
\centering{}\includegraphics[width=2in]{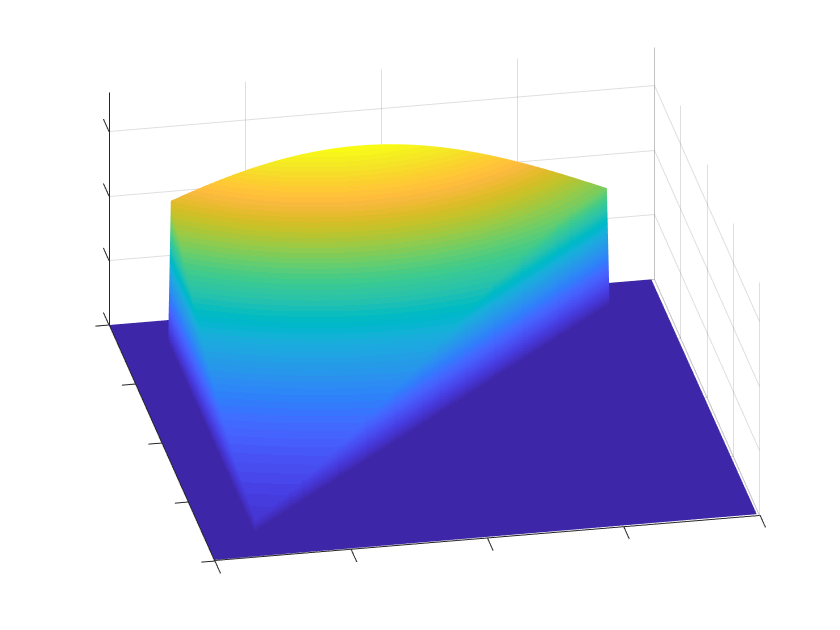}\includegraphics[width=2in]{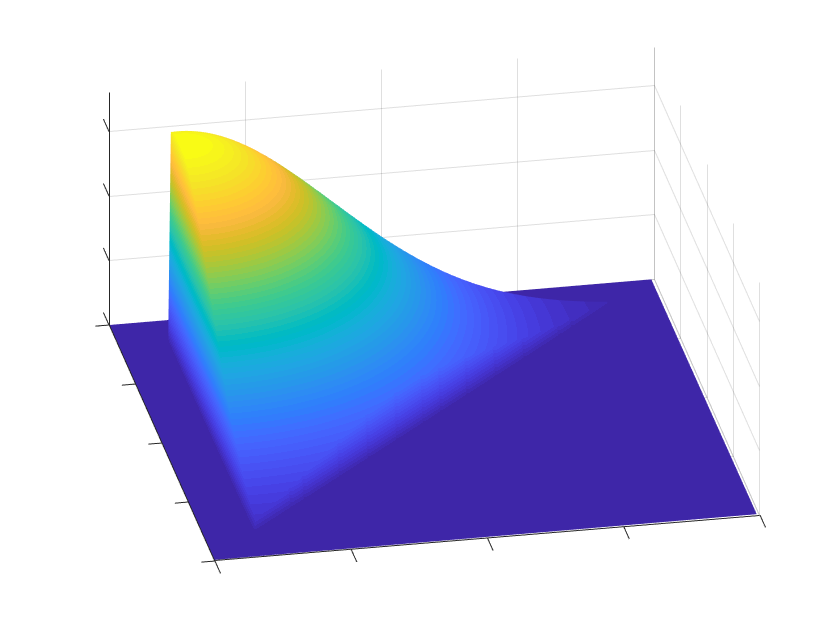}\includegraphics[width=2in]{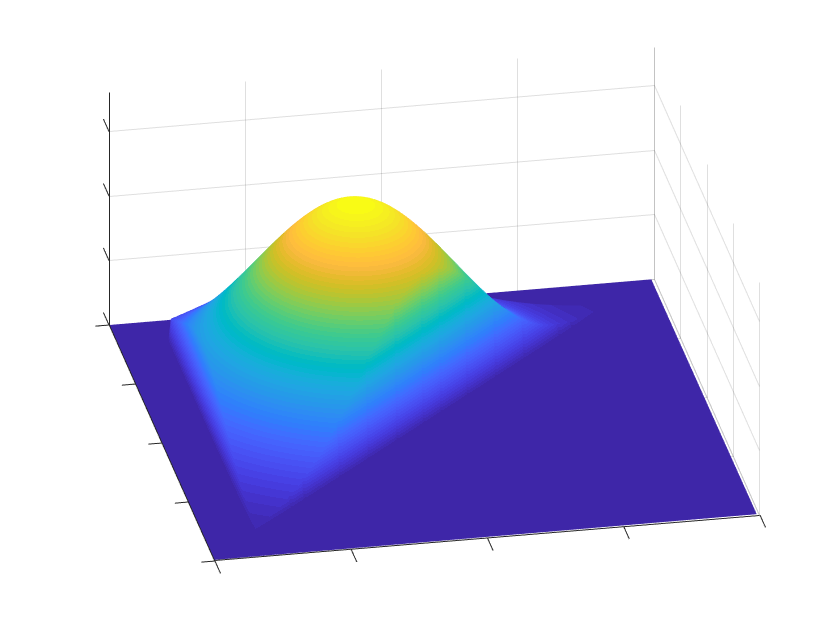}\caption{Stochastic localization}
\end{figure}

\subsubsection{A stochastic process and its properties}

Given a distribution with logconcave density $p(x)$, we start at
time $t=0$ with this distribution and at each time $t>0$, we apply
an infinitesimal change to the density. To make some proofs easier,
one may assume that the support of $p$ is contained in a ball of
radius $n$ because there is only exponentially small probability
outside this ball, at most $e^{-\Omega(n)}$. Let $dW_{t}$ be the
infinitesimal Wiener process.
\begin{defn}
\label{def:A}Given a logconcave distribution $p$, we define the
following stochastic differential equation:
\begin{equation}
c_{0}=0,\quad dc_{t}=dW_{t}+\mu_{t}dt,\label{eq:dBt}
\end{equation}
where the probability distribution $p_{t}$ and its mean $\mu_{t}$
are defined by
\[
p_{t}(x)=\frac{e^{c_{t}^{\top}x-\frac{t}{2}\norm x_{2}^{2}}p(x)}{\int_{\Rn}e^{c_{t}^{\top}y-\frac{t}{2}\norm y_{2}^{2}}p(y)dy},\quad\mu_{t}=\E_{x\sim p_{t}}x.
\]
We will presently explain why $p_{t}$ takes this form with a Gaussian
component. Before we do that, we note that the process can be generalized
using a ``control'' matrix $C_{t}$ at time $t$. This is a positive
definite matrix that could, for example, be used to adapt the process
to the covariance of the current distribution. At time $t$, the covariance
matrix is
\[
A_{t}\defeq\E_{x\sim p_{t}}(x-\mu_{t})(x-\mu_{t})^{\top}
\]
The control matrix is incorporated in the following more general version
of (\ref{eq:dBt}):
\[
c_{0}=0,\quad dc_{t}=C_{t}^{1/2}dW_{t}+C_{t}\mu_{t}dt,
\]
\[
B_{0}=0,\quad dB_{t}=C_{t}dt,
\]
where the probability distribution $p_{t}$ is now defined by 
\[
p_{t}(x)=\frac{e^{c_{t}^{\top}x-\frac{1}{2}\norm x_{B_{t}}^{2}}p(x)}{\int_{\Rn}e^{c_{t}^{\top}y-\frac{1}{2}\norm y_{B_{t}}^{2}}p(y)dy}.
\]
When $C_{t}$ is a Lipschitz function with respect to $c_{t},\mu_{t},A_{t}$
and $t$, standard theorems (e.g., \cite[Sec 5.2]{oksendal2013stochastic})
show the existence and uniqueness of the solution in time $[0,T]$
for any $T>0$. 
\end{defn}

We will now focus on the case $C_{t}=I$ and hence $B_{t}=tI$. The
lemma below says that the stochastic process is the same as continuously
multiplying $p_{t}(x)$ by a random infinitesimally small linear function. 
\begin{lem}
\label{lem:def-pt}We have that $dp_{t}(x)=(x-\mu_{t})^{\top}dW_{t}\cdot p_{t}(x)$
for any $x\in\Rn$. 
\end{lem}

\subsubsection{Alternative definition of the process}

Here, we use $dp_{t}(x)=(x-\mu_{t})^{\top}dW_{t}p_{t}(x)$ as the
definition of the process and show how the Gaussian term $-\frac{t}{2}\norm x_{2}^{2}$
emerges. To compute $d\log p_{t}(x)$, we first explain how to apply
the chain rule (Itô's formula) for a stochastic processes. Given real-valued
stochastic processes $x_{t}$ and $y_{t}$, the quadratic variations
$[x]_{t}$ and $[x,y]_{t}$ are real-valued stochastic processes defined
by 
\begin{align*}
[x]_{t} & =\lim_{|P|\rightarrow0}\sum_{n=1}^{\infty}\left(x_{\tau_{n}}-x_{\tau_{n-1}}\right)^{2},\\{}
[x,y]_{t} & =\lim_{|P|\rightarrow0}\sum_{n=1}^{\infty}\left(x_{\tau_{n}}-x_{\tau_{n-1}}\right)\left(y_{\tau_{n}}-y_{\tau_{n-1}}\right),
\end{align*}
where $P=\{0=\tau_{0}\leq\tau_{1}\leq\tau_{2}\leq\cdots\uparrow t\}$
is a stochastic partition of the non-negative real numbers, $|P|=\max_{n}\left(\tau_{n}-\tau_{n-1}\right)$
is called the \emph{mesh} of $P$ and the limit is defined using convergence
in probability. Note that $[x]_{t}$ is non-decreasing with $t$ and
$[x,y]_{t}$ can be defined as
\[
[x,y]_{t}=\frac{1}{4}\left([x+y]_{t}-[x-y]_{t}\right).
\]
For example, if the processes $x_{t}$ and $y_{t}$ satisfy the SDEs
\[
dx_{t}=\mu(x_{t})dt+\sigma(x_{t})dW_{t}\quad\mbox{and\quad}y_{t}=\nu(y_{t})dt+\eta(y_{t})dW_{t}
\]
where $W_{t}$ is a Wiener process, we have that 
\[
[x]_{t}=\int_{0}^{T}\sigma^{2}(x_{s})ds,\quad[x,y]_{t}=\int_{0}^{T}\sigma(x_{s})\eta(y_{s})ds\quad\mbox{and}\quad d[x,y]_{t}=\sigma(x_{t})\eta(y_{t})dt.
\]
For vector-valued SDEs 
\[
dx_{t}=\mu(x_{t})dt+\Sigma(x_{t})dW_{t}\quad\mbox{and}\quad dy_{t}=\nu(y_{t})dt+M(y_{t})dW_{t}
\]
we have that 
\[
[x^{i},x^{j}]_{t}=\int_{0}^{T}(\Sigma(x_{s})\Sigma^{\top}(x_{s}))_{ij}ds\quad\mbox{and}\quad d[x^{i},y^{j}]_{t}=\int_{0}^{T}(\Sigma(x_{s})M^{\top}(y_{s}))_{ij}ds.
\]
\begin{lem}[Itô's formula]
 Let $x$ be a semimartingale and $f$ be a twice continuously differentiable
function, then
\[
df(x_{t})=\sum_{i}\frac{df(x_{t})}{dx^{i}}dx^{i}+\frac{1}{2}\sum_{i,j}\frac{d^{2}f(x_{t})}{dx^{i}dx^{j}}d[x^{i},x^{j}]_{t}.
\]
\end{lem}

We can now compute the derivative $d\log p_{t}(x)$. Using $dp_{t}(x)$
in Lemma \ref{lem:def-pt} and Itô's formula, we have that

\begin{align*}
d\log p_{t}(x) & =\frac{dp_{t}(x)}{p_{t}(x)}-\frac{1}{2}\frac{d[p_{t}(x)]_{t}}{p_{t}(x)^{2}}\\
 & =(x-\mu_{t})^{\top}dW_{t}-\frac{1}{2}(x-\mu_{t})^{\top}(x-\mu_{t})dt\\
 & =x^{\top}(dW_{t}+\mu_{t}dt)-\frac{1}{2}\norm x^{2}dt+g(t)\\
 & =x^{\top}dc_{t}-\frac{1}{2}\norm x^{2}dt+g(t)
\end{align*}
where $g(t)$ represents terms independent of $x$ and the first two
terms explain the form of $p_{t}(x)$ and the appearance of the Gaussian. 

Using Itô's formula again, we can compute the change of the covariance
matrix:

\[
dA_{t}=\int_{\Rn}(x-\mu_{t})(x-\mu_{t})^{\top}\cdot(x-\mu_{t})^{\top}dW_{t}\cdot p_{t}(x)dx-A_{t}^{2}dt.
\]

\subsubsection{Applications}

\paragraph{Better KLS bound.}

Eldan introduced stochastic localization \cite{Eldan2013} and used
it to prove that the thin-shell conjecture implies the KLS conjecture
up to a logarithmic factors. We later adapted his idea to get a better
bound on KLS constant itself. Since our proof is slightly simpler
and more direct than Eldan's proof, we choose to only discuss our
proof in full detail. We will discuss the difference between our proof
and his proof in the next section.
\begin{thm}[\cite{LeeV17KLS}]
\label{thm:n14} For any logconcave density $p$ in $\R^{n}$ with
covariance matrix $A$,
\[
\psi_{p}\gtrsim\frac{1}{\left(\tr\left(A^{2}\right)\right)^{1/4}}.
\]
In particular, we have $\psi_{p}\gtrsim n^{-\frac{1}{4}}$ for any
isotropic logconcave $p$.
\end{thm}

We now outline the proof. For this, we use the simplest choice in
stochastic localization, namely a pure random direction chosen from
the uniform distribution (i.e., $C_{t}=I$). The analysis needs a
potential function that grows slowly but still maintains good control
over the spectral norm of the current covariance matrix. The direct
choice of $\norm{A_{t}}_{\op}$, where $A_{t}$ is the covariance
matrix of the distribution at time $t$, is hard to control. We use
the potential $\Phi_{t}=\tr(A_{t}^{2})$. Itô's formula shows that
this function evolves as follows:
\begin{align}
d\Phi_{t}= & -2\tr(A_{t}^{3})dt+\E_{x,y\sim p_{t}}\left((x-\mu_{t})^{\top}(y-\mu_{t})\right)^{3}dt+2\E_{x\sim p_{t}}(x-\mu_{t})^{\top}A_{t}(x-\mu_{t})(x-\mu_{t})^{\top}dW_{t}\nonumber \\
\defeq & \delta_{t}dt+v_{t}^{\top}dW_{t}.\label{eq:dPhi}
\end{align}

The first term can be viewed as a deterministic drift while the second
is stochastic with no bias. To bound both terms, we use the following
lemmas. The first one below is a folklore reverse Hölder inequality
and can be proved using the localization lemma (see e.g., \cite{LV07}).
\begin{lem}[Logconcave moments]
\label{lem:lcmom} Given a logconcave density $p$ in $\R^{n}$,
and any positive integer $k$, 
\[
\E_{x\sim p}\norm x^{k}\le(2k)^{k}\left(\E_{x\sim p}\norm x^{2}\right)^{k/2}.
\]
\end{lem}

Using this lemma and the Cauchy-Schwarz inequality, we have the following
moment bounds.
\begin{lem}
\label{lem:key}Given a logconcave distribution $p$ with mean $\mu$
and covariance $A$, 
\begin{enumerate}
\item $\E_{x,y\sim p}\left|\left\langle x-\mu,y-\mu\right\rangle \right|^{3}\lesssim\tr\left(A^{2}\right)^{3/2}.$
\item $\norm{\E_{x\sim p}(x-\mu)(x-\mu)^{\top}A(x-\mu)}_{2}\lesssim\norm A_{\op}^{1/2}\tr\left(A^{2}\right).$
\end{enumerate}
\end{lem}

\begin{proof}
Without loss of generality, we can assume $\mu=0$. 

For the first statement, we fix $x$ and apply Lemma \ref{lem:lcmom}
to show that
\[
\E_{y\sim p}\left|\left\langle x,y\right\rangle \right|^{3}\lesssim\left(\E_{y\sim p}\langle x,y\rangle^{2}\right)^{3/2}=\left(x^{\top}Ax\right)^{3/2}=\norm{A^{1/2}x}_{2}^{3}.
\]
Then we note that $A^{1/2}x$ follows a logconcave distribution (Lemma
\ref{lem:marginal}) with mean $0$ and covariance $A^{2}$ and hence
Lemma \ref{lem:lcmom} to see that 
\[
\E_{x\sim p}\norm{A^{1/2}x}_{2}^{3}\lesssim\left(\E_{x\sim p}\norm{A^{1/2}x}^{2}\right)^{3/2}=\tr\left(A^{2}\right)^{3/2}.
\]
Therefore, we have that 
\[
\E_{x,y\sim p}\left|\left\langle x,y\right\rangle \right|^{3}\lesssim\E_{x\sim p}\norm{A^{1/2}x}^{3}\lesssim\tr\left(A^{2}\right)^{3/2}.
\]

For the second statement,
\begin{align*}
\norm{\E_{x\sim p}x\cdot x^{\top}Ax}_{2} & =\max_{\norm{\zeta}_{2}\leq1}\E_{x\sim p}x^{\top}\zeta\cdot x^{\top}Ax\\
 & \leq\max_{\norm{\zeta}_{2}\leq1}\sqrt{\E_{x\sim p}(x^{\top}\zeta)^{2}}\sqrt{\E_{x\sim p}\left(x^{\top}Ax\right)^{2}}\\
 & =\norm A_{\op}^{1/2}\cdot\sqrt{\E_{x\sim p}\norm{A^{1/2}x}_{2}^{4}}.
\end{align*}
For the last term, by a similar argument as before, we can use Lemma
\ref{lem:lcmom} shows that 
\[
\E_{x\sim p}\norm{A^{\frac{1}{2}}x}_{2}^{4}\lesssim\left(\tr A^{2}\right)^{2}.
\]
This gives the second statement.
\end{proof}
The drift term in (\ref{eq:dPhi}) can be bounded using the first
inequality in Lemma \ref{lem:key} as 
\begin{equation}
\delta_{t}\le\E_{x,y\sim p_{t}}\left((x-\mu_{t})^{\top}(y-\mu_{t})\right)^{3}\lesssim\tr\left(A_{t}^{2}\right)^{3/2}=\Phi_{t}^{3/2}\label{eq:alpha_bound}
\end{equation}
where we also used that the term $-2\tr(A_{t}^{3})$ is negative since
$A_{t}$ is positive semi-definite. The martingale coefficient $v_{t}$
can be bounded in magnitude using the second inequality:
\[
\norm{v_{t}}_{2}\le\norm{\E_{x\sim p_{t}}(x-\mu_{t})^{\top}A_{t}(x-\mu_{t})(x-\mu_{t})}_{2}\le\norm{A_{t}}_{\op}^{1/2}\tr(A_{t}^{2})\apprle\Phi_{t}^{5/4}.
\]
Together we have the simplified expression 
\[
d\Phi_{t}\apprle\Phi_{t}^{3/2}dt+\Phi_{t}^{5/4}dW_{t}
\]
So the drift term grows roughly as $\Phi^{3/2}t$ while the stochastic
term grows as $\Phi_{t}^{5/4}\sqrt{t}.$ Thus, both bounds indicate
that for $t$ up to $O(\frac{1}{\sqrt{\tr A^{2}}})$, the potential
$\Phi_{t}$ remains $O(\tr A^{2})$, i.e., $\tr(A_{t}^{2})$ grows
only by a constant factor. 

We can use this as follows. Fix any subset $E\subset\Rn$ of measure
$p(E)=\int_{E}p(x)dx=\frac{1}{2}$. We will argue that the set remains
nearly balanced for a while. To see this, let $g_{t}=p_{t}(E)$ and
note that
\[
dg_{t}=\left\langle \int_{E}(x-\mu_{t})p_{t}(x)dx,dW_{t}\right\rangle 
\]
where the integral might not be $0$ because it is over the subset
$E$ and not all of $\R^{n}$. Hence, 
\begin{align*}
d[g_{t}]_{t} & =\norm{\int_{E}(x-\mu_{t})p_{t}(x)dx}_{2}^{2}dt\\
 & =\max_{\norm{\zeta}_{2}\leq1}\left(\int_{E}\zeta^{\top}(x-\mu_{t})p_{t}(x)dx\right)^{2}dt\\
 & \leq\max_{\norm{\zeta}_{2}\leq1}\int_{\Rn}\left(\zeta^{\top}(x-\mu_{t})\right)^{2}p_{t}(x)dx\cdot p_{t}(E)dt\\
 & \leq\max_{\norm{\zeta}_{2}\leq1}\left(\zeta^{\top}A_{t}\zeta\right)dt=\norm{A_{t}}_{\op}dt.
\end{align*}
Thus, $g_{t}$ is bounded by a random process with variance $\norm{A_{t}}_{\op}$
at time $t$. For $0\leq T\apprle\frac{1}{\sqrt{\tr A^{2}}}$, the
total variance accumulated in the time period $[0,T]$ is 
\[
\int_{0}^{T}\norm{A_{s}}_{\op}ds\lesssim\int_{0}^{T}\tr(A_{s}^{2})^{\frac{1}{2}}ds\lesssim1.
\]
Hence, we get that the set $E$ remains bounded in measure between
$\frac{1}{4}$ and $\frac{3}{4}$ till time $T=\frac{c}{\sqrt{\tr A^{2}}}$
for some universal constant $c$. 

But at this time, the density $p_{T}$ has a Gaussian component with
coefficient $T$ and hence the Cheeger constant is $\Omega(\sqrt{T})$
by Theorem \ref{thm:Gaussian-iso}. Hence, we have the following:
\begin{align*}
p(\partial E) & =\E p_{T}(\partial E)\\
 & \gtrsim\sqrt{T}\cdot\P(\frac{1}{4}\leq p_{T}(E)\leq\frac{3}{4})\\
 & \gtrsim\sqrt{T}\\
 & \gtrsim(\tr A^{2})^{-\frac{1}{4}}
\end{align*}
where the first equality follows from the fact that $p_{t}$ is a
martingale, the second inequality follows from $\psi_{p_{T}}\gtrsim\sqrt{T}$
(Theorem \ref{thm:Gaussian-iso}) and the third inequality follows
from the fact the set $E$ remains bounded in measure between $\frac{1}{4}$
and $\frac{3}{4}$ till time $T=\frac{c}{\sqrt{\tr A^{2}}}$ with
at least constant probability. This completes the proof of Theorem
\ref{thm:n14}.

\paragraph{Reduction from thin shell to KLS.}

In Eldan's original proof, he used the stochastic process with control
matrix $C_{t}=A_{t}^{-\frac{1}{2}}$. For this, the change of the
covariance matrix is as follows:

\[
dA_{t}=\int_{\Rn}(x-\mu_{t})(x-\mu_{t})^{\top}\cdot(x-\mu_{t})^{\top}A_{t}^{-\frac{1}{2}}dW_{t}\cdot p_{t}(x)dx-A_{t}dt.
\]
To get the reduction from thin shell to KLS, we use the potential
$\Phi_{t}=\tr A_{t}^{q}$ with a suitably large integer $q$ for better
control of the spectral norm. 
\begin{align*}
d\Phi_{t}= & -q\Phi_{t}dt+q\E_{x\sim p_{t}}(x-\mu_{t})^{\top}A_{t}^{q-1}(x-\mu_{t})(x-\mu_{t})^{\top}A_{t}^{-\frac{1}{2}}dW_{t}\\
 & +\frac{q}{2}\sum_{\alpha+\beta=q-2}\E_{x,y\sim p_{t}}(x-\mu_{t})^{\top}A_{t}^{\alpha}(y-\mu_{t})(x-\mu_{t})^{\top}A_{t}^{\beta}(y-\mu_{t})(x-\mu_{t})^{\top}A_{t}^{-1}(y-\mu_{t})dt.
\end{align*}
The stochastic term can be bounded using the same proof as in the
second inequality of Lemma \ref{lem:key}. However, the last term
is more complicated. Eldan used the thin shell conjecture to bound
the last term and showed that
\[
d\Phi_{t}\apprle q\Phi_{t}dW_{t}+q^{2}\sigma(n)^{2}\log n\cdot\Phi_{t}dt
\]
where $\sigma(n)=\sup_{p}\sigma_{p}$ is the maximum thin-shell constant
over all isotropic logconcave densities $p$ in $\Rn$.

For an isotropic distribution, $\Phi_{0}=n$. Hence, we have that
$\Phi_{t}\apprle n^{O(1)}$ for $0\leq t\leq\frac{1}{q^{2}\sigma(n)^{2}}$.
By choosing $q=O(\log n)$, we have that $\norm{A_{t}}_{\op}\apprle1$
for $0\leq t\leq\frac{1}{\sigma(n)^{2}\log^{2}n}$. By a similar proof
as before, this gives that $\psi_{p}\gtrsim\frac{1}{\sigma(n)\log n}$.
\begin{thm}[\cite{Eldan2013}]
\label{thm:eldan}For any isotropic logconcave density $p$ in $\R^{n}$,
we have $\psi_{p}\gtrsim\frac{1}{\sigma(n)\log n}$ where $\sigma(n)=\sup_{p}\sigma_{p}$
is the maximum thin-shell constant over all isotropic logconcave densities
$p$ in $\Rn$.
\end{thm}

\paragraph{Tight log-Sobolev constant and improved concentration.}

One can view the slicing conjecture as being weaker than the thin-shell
conjecture and the thin shell conjecture as weaker than the KLS conjecture.
Naturally one may ask if there is a conjecture stronger than the KLS
conjecture. It is known that KLS conjecture is equivalent to proving
Poincaré constant is $\Theta(1)$ for any isotropic logconcave distribution.
It is also known that log-Sobolev constant (defined below) is stronger
than the Poincaré constant. So, a natural question is whether the
log-Sobolev constant is $\Theta(1)$ for any isotropic logconcave
distribution? We first remind the reader of the definition.
\begin{defn}
For any distribution $p$, we define the the log-Sobolev constant
$\rho_{p}$ be the largest number such that for every smooth $f$
with $\int f^{2}(x)p(x)dx=1$, we have that
\[
\int\left|\nabla f(x)\right|^{2}p(x)dx\gtrsim\rho_{p}\int f^{2}(x)\log f(x)^{2}\cdot p(x)dx.
\]
\end{defn}

The result of \cite{KannanLM06} (Theorem \ref{thm:-The-log-Cheeger})
shows that $\rho_{p}\ge\frac{1}{D^{2}}$ for any isotropic logconcave
measure with support of diameter $D$. Recently, we proved the following
tight bound.
\begin{thm}[\cite{lee2017stochastic}]
For any isotropic logconcave density $p$ in $\R^{n}$ with support
of diameter $D$, the log-Sobolev constant satisfies $\rho_{p}\gtrsim\frac{1}{D}$.
This is the best possible bounds up to a constant. 
\end{thm}

The proof uses the same process as Theorem \ref{thm:n14} with a different
potential function that allows one to get more control on $\norm{A_{t}}_{\op}$.
This potential is a Stieltjes-type barrier function defined as follows.
Let $u(A_{t})$ be the solution to
\begin{equation}
\tr((uI-A_{t})^{-2})=n\text{ and }A_{t}\preceq uI\label{eq:def_u}
\end{equation}
Note that this is the same as $\sum_{i=1}^{n}\frac{1}{(u-\lambda_{i})^{2}}=n$
and $\lambda_{i}\leq u$ for all $i$ where $\lambda_{i}$ are the
eigenvalues of $A_{t}$. Such a potential was used to to build graph
sparsifiers \cite{BSS12,allen2015spectral,lee2015constructing,lee2017sdp},
to understand covariance estimation \cite{srivastava2013covariance}
and to solve bandit problems \cite{audibert2013regret}.

The next theorem is a large-deviation inequality based on the same
proof technique. 
\begin{thm}[\cite{lee2017stochastic}]
\label{thm:conc}For any $L$-Lipschitz function $g$ in $\Rn$ and
any isotropic logconcave density $p$, we have that
\[
\P_{x\sim p}\left(\left|g(x)-\mathrm{med}_{x\sim p}g(x)\right|\geq c\cdot L\cdot t\right)\leq\exp(-\frac{t^{2}}{t+\sqrt{n}}).
\]
Furthermore, the same conclusion holds with $\mathrm{med}_{x\sim p}g(x)$
replaced by $\E_{x\sim p}g(x)$. 
\end{thm}

For the Euclidean norm $g(x)=\norm x,$ the range $t\ge\sqrt{n}$
is a well-known inequality proved by Paouris \cite{Paouris2006} and
later refined by Guedon and Milman \cite{GuedonM11} to $\exp(-\min(\frac{t^{3}}{n},t$)).
The bound above improves and generalizes these bounds. 

\subsection{Needle decompositions}

We describe a ``combinatorial'' approach to resolving the KLS conjecture.
The idea of localization was to reduce an isoperimetric inequality
in $\R^{n}$ to a similar inequality in one dimension, by arguing
that if the original inequality were false, there would be a one-dimensional
counterexample. Alternatively, one can view localization as an inductive
process \textemdash{} the final inequality is a weighted sum of inequalities
for each component of a partition into needles, viz. a \emph{needle
decomposition}. For this to be valid, the partition should maintain
the relative measure of the subset $S$ whose isoperimetry is being
considered. To be useful, the Cheeger constant of each needle should
be approximately as large as desired. In fact, it suffices if some
constant fraction of needles (by measure) in a needle decomposition
has good isoperimetry, i.e., small spectral norm, i.e., variance equal
to $O(\norm A_{\op})$. 
\begin{defn}
An $\epsilon$-thin cylinder decomposition of a convex body $K$ is
a partition of $K$ by hyperplane cuts so that each part $P$ is contained
in a cyclinder whose radius is at most $\epsilon.$ The limit of a
sequence of needle decompositions with $\epsilon\rightarrow0$ is
a needle decomposition with weighting $w(P)$ over the limiting set
of needles ${\bf P}$. 
\end{defn}

\begin{figure}
\centering{} \begin{tikzpicture}[y=0.80pt, x=0.80pt, yscale=-0.400000, xscale=0.400000, inner sep=0pt, outer sep=0pt]   \path[draw=cff0000,line join=miter,line cap=butt,miter limit=4.00,even odd     rule,line width=1.000pt] (159.3288,313.4491) .. controls (438.7337,257.8456)     and (438.7337,257.8456) .. (438.7337,257.8456) -- (615.7436,293.1382) --     (676.0063,435.8639) -- (561.9376,566.5282) -- (252.6479,608.3066) --     (90.5970,556.4771) -- (111.0267,402.8090) -- cycle;   \path[draw=black,line join=miter,line cap=butt,miter limit=4.00,even odd     rule,line width=1.000pt] (237.5287,347.7262) .. controls (300.4217,347.5116)     and (300.4217,347.5116) .. (300.4217,347.5116);   \path[draw=c008600,line join=miter,line cap=butt,miter limit=4.00,even odd     rule,line width=1.000pt] (338.7088,412.0355) .. controls (445.2011,411.8210)     and (445.2011,411.8210) .. (445.2011,411.8210);   \path[draw=black,line join=miter,line cap=butt,miter limit=4.00,even odd     rule,line width=1.000pt] (283.4969,333.5692) .. controls (349.6196,333.3547)     and (349.6196,333.3547) .. (349.6196,333.3547);   \path[draw=black,line join=miter,line cap=butt,miter limit=4.00,even odd     rule,line width=1.000pt] (533.7884,478.0702) .. controls (643.3930,426.7771)     and (643.3930,426.7771) .. (643.3930,426.7771);   \path[draw=black,line join=miter,line cap=butt,miter limit=4.00,even odd     rule,line width=1.000pt] (204.4979,439.3463) .. controls (279.6222,418.9633)     and (279.6222,418.9633) .. (279.6222,418.9633);   \path[draw=black,line join=miter,line cap=butt,miter limit=4.00,even odd     rule,line width=1.000pt] (165.4128,525.5287) .. controls (286.0479,517.4029)     and (286.0479,517.4029) .. (286.0479,517.4029);   \path[draw=black,line join=miter,line cap=butt,miter limit=4.00,even odd     rule,line width=1.000pt] (220.9908,393.8070) .. controls (303.2613,393.5925)     and (303.2613,393.5925) .. (303.2613,393.5925);   \path[draw=black,line join=miter,line cap=butt,miter limit=4.00,even odd     rule,line width=1.000pt] (351.1261,498.8160) .. controls (292.6770,423.6296)     and (292.6770,423.6296) .. (292.6770,423.6296);   \path[draw=black,line join=miter,line cap=butt,miter limit=4.00,even odd     rule,line width=1.000pt] (392.1792,342.7931) .. controls (491.6554,310.6852)     and (491.6554,310.6852) .. (491.6554,310.6852);   \path[draw=black,line join=miter,line cap=butt,miter limit=4.00,even odd     rule,line width=1.000pt] (449.7815,348.6541) .. controls (493.2212,387.1712)     and (493.2212,387.1712) .. (493.2212,387.1712);   \path[draw=black,line join=miter,line cap=butt,miter limit=4.00,even odd     rule,line width=1.000pt] (368.7285,319.6284) .. controls (457.2148,269.5831)     and (457.2148,269.5831) .. (457.2148,269.5831);   \path[draw=black,line join=miter,line cap=butt,miter limit=4.00,even odd     rule,line width=1.000pt] (424.0562,347.6251) .. controls (616.9442,312.8997)     and (616.9442,312.8997) .. (616.9442,312.8997);   \path[draw=black,line join=miter,line cap=butt,miter limit=4.00,even odd     rule,line width=1.000pt] (317.4839,385.4792) .. controls (381.3588,355.2153)     and (381.3588,355.2153) .. (381.3588,355.2153);   \path[draw=black,line join=miter,line cap=butt,miter limit=4.00,even odd     rule,line width=1.000pt] (385.9947,470.5145) .. controls (475.6819,422.0183)     and (475.6819,422.0183) .. (475.6819,422.0183);   \path[draw=black,line join=miter,line cap=butt,miter limit=4.00,even odd     rule,line width=1.000pt] (262.0695,372.7280) .. controls (422.4688,316.5376)     and (422.4688,316.5376) .. (422.4688,316.5376);   \path[draw=black,line join=miter,line cap=butt,miter limit=4.00,even odd     rule,line width=1.000pt] (353.0134,383.2598) .. controls (464.8631,373.4093)     and (464.8631,373.4093) .. (464.8631,373.4093);   \path[draw=black,line join=miter,line cap=butt,miter limit=4.00,even odd     rule,line width=1.000pt] (335.3379,487.7370) .. controls (260.9084,552.0656)     and (260.9084,552.0656) .. (260.9084,552.0656);   \path[draw=black,line join=miter,line cap=butt,miter limit=4.00,even odd     rule,line width=1.000pt] (444.5515,573.6872) .. controls (506.1596,552.3675)     and (506.1596,552.3675) .. (506.1596,552.3675);   \path[draw=black,line join=miter,line cap=butt,miter limit=4.00,even odd     rule,line width=1.000pt] (318.5603,444.2673) .. controls (396.3588,435.6986)     and (396.3588,435.6986) .. (396.3588,435.6986);   \path[draw=black,line join=miter,line cap=butt,miter limit=4.00,even odd     rule,line width=1.000pt] (345.2995,458.7922) .. controls (443.5959,493.3954)     and (443.5959,493.3954) .. (443.5959,493.3954);   \path[draw=black,line join=miter,line cap=butt,miter limit=4.00,even odd     rule,line width=1.000pt] (354.7512,573.8244) .. controls (406.1791,556.4193)     and (406.1791,556.4193) .. (406.1791,556.4193);   \path[draw=black,line join=miter,line cap=butt,miter limit=4.00,even odd     rule,line width=1.000pt] (214.7654,569.1847) .. controls (318.4995,548.8455)     and (318.4995,548.8455) .. (318.4995,548.8455);   \path[draw=black,line join=miter,line cap=butt,miter limit=4.00,even odd     rule,line width=1.000pt] (319.1600,521.4980) .. controls (444.1604,546.4665)     and (444.1604,546.4665) .. (444.1604,546.4665);   \path[draw=black,line join=miter,line cap=butt,miter limit=4.00,even odd     rule,line width=1.000pt] (104.0714,542.3710) .. controls (241.2650,544.3681)     and (241.2650,544.3681) .. (241.2650,544.3681);   \path[draw=black,line join=miter,line cap=butt,miter limit=4.00,even odd     rule,line width=1.000pt] (425.9993,461.2306) .. controls (503.7977,440.8914)     and (503.7977,440.8914) .. (503.7977,440.8914);   \path[draw=black,line join=miter,line cap=butt,miter limit=4.00,even odd     rule,line width=1.000pt] (569.6178,395.4914) .. controls (651.0052,406.4993)     and (651.0052,406.4993) .. (651.0052,406.4993);   \path[draw=black,line join=miter,line cap=butt,miter limit=4.00,even odd     rule,line width=1.000pt] (481.6100,402.7484) .. controls (583.7106,428.3127)     and (583.7106,428.3127) .. (583.7106,428.3127);   \path[draw=black,line join=miter,line cap=butt,miter limit=4.00,even odd     rule,line width=1.000pt] (425.4699,578.3279) .. controls (491.9579,506.0897)     and (491.9579,506.0897) .. (491.9579,506.0897);   \path[draw=black,line join=miter,line cap=butt,miter limit=4.00,even odd     rule,line width=1.000pt] (109.4154,520.1002) .. controls (158.9570,502.0380)     and (158.9570,502.0380) .. (158.9570,502.0380);   \path[draw=black,line join=miter,line cap=butt,miter limit=4.00,even odd     rule,line width=1.000pt] (358.1925,513.4731) .. controls (439.8454,504.4257)     and (439.8454,504.4257) .. (439.8454,504.4257);   \path[draw=black,line join=miter,line cap=butt,miter limit=4.00,even odd     rule,line width=1.000pt] (443.8774,295.1846) .. controls (505.2203,309.6724)     and (505.2203,309.6724) .. (505.2203,309.6724);   \path[draw=black,line join=miter,line cap=butt,miter limit=4.00,even odd     rule,line width=1.000pt] (147.5421,344.9744) .. controls (238.3083,324.6352)     and (238.3083,324.6352) .. (238.3083,324.6352);   \path[draw=black,line join=miter,line cap=butt,miter limit=4.00,even odd     rule,line width=1.000pt] (210.1887,307.9176) .. controls (357.0506,307.7030)     and (357.0506,307.7030) .. (357.0506,307.7030);   \path[draw=black,line join=miter,line cap=butt,miter limit=4.00,even odd     rule,line width=1.000pt] (208.4600,461.7218) .. controls (326.7170,484.0562)     and (326.7170,484.0562) .. (326.7170,484.0562);   \path[draw=black,line join=miter,line cap=butt,miter limit=4.00,even odd     rule,line width=1.000pt] (192.2728,452.3280) .. controls (298.7650,452.1134)     and (298.7650,452.1134) .. (298.7650,452.1134);   \path[draw=black,line join=miter,line cap=butt,miter limit=4.00,even odd     rule,line width=1.000pt] (104.4534,474.3479) .. controls (194.6713,470.6180)     and (194.6713,470.6180) .. (194.6713,470.6180);   \path[draw=black,line join=miter,line cap=butt,miter limit=4.00,even odd     rule,line width=1.000pt] (123.2090,398.7903) .. controls (246.2683,412.5877)     and (246.2683,412.5877) .. (246.2683,412.5877);   \path[draw=black,line join=miter,line cap=butt,miter limit=4.00,even odd     rule,line width=1.000pt] (182.7549,583.2892) .. controls (319.8558,578.5484)     and (319.8558,578.5484) .. (319.8558,578.5484);   \path[draw=black,line join=miter,line cap=butt,miter limit=4.00,even odd     rule,line width=1.000pt] (441.0146,468.8532) .. controls (544.8322,488.8861)     and (544.8322,488.8861) .. (544.8322,488.8861);   \path[draw=black,line join=miter,line cap=butt,miter limit=4.00,even odd     rule,line width=1.000pt] (480.7703,491.0400) .. controls (536.1490,512.6770)     and (536.1490,512.6770) .. (536.1490,512.6770);   \path[draw=black,line join=miter,line cap=butt,miter limit=4.00,even odd     rule,line width=1.000pt] (142.7883,390.4192) .. controls (219.6977,369.2780)     and (219.6977,369.2780) .. (219.6977,369.2780);   \path[draw=black,line join=miter,line cap=butt,miter limit=4.00,even odd     rule,line width=1.000pt] (113.2321,461.2434) .. controls (193.8613,421.2965)     and (193.8613,421.2965) .. (193.8613,421.2965);   \path[draw=black,line join=miter,line cap=butt,miter limit=4.00,even odd     rule,line width=1.000pt] (480.3393,353.2142) .. controls (586.4154,361.1615)     and (586.4154,361.1615) .. (586.4154,361.1615);   \path[draw=black,line join=miter,line cap=butt,miter limit=4.00,even odd     rule,line width=1.000pt] (514.2229,386.8430) .. controls (640.1920,371.3402)     and (640.1920,371.3402) .. (640.1920,371.3402);   \path[draw=black,line join=miter,line cap=butt,miter limit=4.00,even odd     rule,line width=1.000pt] (573.2532,516.8102) .. controls (622.3326,461.5206)     and (622.3326,461.5206) .. (622.3326,461.5206);   \begin{scope}[cm={{0.61292,0.0,0.0,0.33698,(238.44856,354.11488)}},draw=c008600,miter limit=4.00,line width=1.000pt]     \path[draw=c008600,line cap=butt,miter limit=4.00,fill opacity=0.761,nonzero       rule,line width=1.000pt] (146.2963,174.5844) ellipse (0.7317cm and 1.6202cm);     \path[draw=c008600,line cap=butt,miter limit=4.00,fill opacity=0.761,nonzero       rule,line width=1.000pt] (335.7073,175.5591) ellipse (0.4965cm and 0.7839cm);     \path[draw=c008600,line join=miter,line cap=butt,miter limit=4.00,even odd       rule,line width=1.000pt] (146.0805,116.8568) -- (335.1852,147.0294);     \path[draw=c008600,line join=miter,line cap=butt,miter limit=4.00,even odd       rule,line width=1.000pt] (145.5166,232.2355) -- (338.4851,203.3369);   \end{scope}   \path[draw=black,line join=miter,line cap=butt,miter limit=4.00,even odd     rule,line width=1.000pt] (477.1348,528.9992) .. controls (571.8727,536.9465)     and (571.8727,536.9465) .. (571.8727,536.9465);   \path[draw=black,line join=miter,line cap=butt,miter limit=4.00,even odd     rule,line width=1.000pt] (187.2066,478.6503) .. controls (258.1924,507.7507)     and (258.1924,507.7507) .. (258.1924,507.7507);
\end{tikzpicture}\caption{Needle decomposition}
\end{figure}
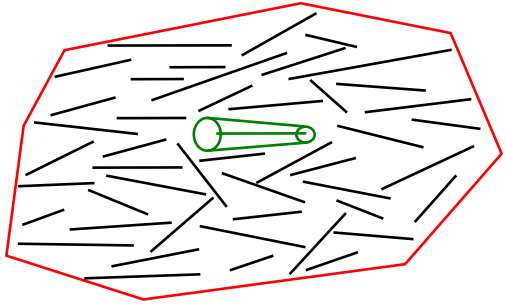
\begin{thm}
Let ${\bf P}$ be a needle decomposition of an isotropic convex body
$K$ by hyperplane cuts s.t. for some $\int_{P}fdx=0$ for each needle
$P$ in the decomposition, where $f:\R^{n}\rightarrow\R$ is a continuous
function. Suppose also that a random needle chosen with density proportional
to $w(P)$ satisfies $\P_{P\in{\bf P}}(\norm{A_{P}}_{\op}\le C)\ge c$,
for constants $c,C>0$. Then $\psi_{K}\gtrsim1$. 
\end{thm}

The proof of the theorem is simple. We fix a subset $S\subset K$
of measure $a\le1/2$ and choose $f$ to be the function which maintains
the fraction of volume taken up $S$ in each part; thus, the relative
measure of $S$ in each needle is $a$. Next, using one-dimensional
isoperimetry, the measure of the boundary of $S$ in each needle $P$
is at least $\Omega(1)\frac{a}{\norm{A_{P}}_{\op}^{1/2}}w(P)$. This
is $\Omega(a)w(P)$ for at least $c$ fraction of the needles (by
their weight) by the assumption of the theorem. Hence 
\[
\frac{\vol_{n-1}(\partial S)}{\vol(S)}\ge c\cdot\Omega(a)=\Omega(a)
\]
showing that $\psi_{K}=\Omega(1)$.

This puts the focus on whether there exist needle decompositions that
have bounded operator norm for some constant measure of the needles.
This approach was used in \cite{CDV10} to bound the isoperimetry
of star-shaped bodies.

As far as we know, the bounded operator norm property might be true
for \emph{any} needle decomposition! We conclude with this as a question: 

Let ${\bf P}$ be a partition of an isotropic logconcave density $p$
in $\R^{n}$ by hyperplanes. Is it true that there always exists a
subset of parts $Q\subset\P$ such that (1) $p(Q)\ge c$ and (2) $\Var(P)\le C$
for each $P\in Q$? ($c,C$ are absolute constants, and $\Var(P)$
is the variance of a random point from the part $P$ drawn with density
proportional to $p$).

\section{Open Problems}

Here we discuss some intriguing open questions related to the KLS
conjecture (besides resolving it!), asymptotic convex geometry and
efficient algorithms in general.

\paragraph{Deterministic volume.}

Efficient algorithms for volume computation are randomized, and this
is unavoidable if access to the input body is only through an oracle
\cite{E86,Barany1987}. However, for explicit polytopes given as $Ax\ge b$,
the only known hardness is for exact computation \cite{DyerF90} and
it is possible that there is a deterministic polynomial-time approximation
scheme. Such an algorithm would be implied if P=BPP. Thus, finding
a deterministic polynomial-time algorithm for estimating the volume
of a polytope to within a factor of $2$ (say) is an outstanding open
problem.

\paragraph{Lower bound for sampling.}

The complexity of sampling an isotropic logconcave density in $\R^{n}$,
assuming the KLS conjecture is $O^{*}(n^{2})$ from a warm start.
Is this the best possible? Can we show an $\Omega(n^{2})$ lower bound
for randomized algorithms?

\paragraph{Faster sampling and isotropy.}

The current bottleneck for faster sampling (in the general oracle
model) is the complexity of isotropic transformation, which is currently
$O^{*}(n^{4})$ \cite{LV2}. Cousins and Vempala have conjectured
that the following algorithm will terminate in $O(\log n)$ iterations
and produce a nearly-istropic body. Each iteration needs $O^{*}(n)$
samples and takes $O^{*}(n^{2})\times O^{*}(n)=O^{*}(n^{3})$ steps.

Repeat: 
\begin{enumerate}
\item Estimate the covariance of the standard Gaussian density restricted
to the current convex body. 
\item If the covariance has eigenvalues smaller than some constant, apply
a transformation to make this distribution isotropic. 
\end{enumerate}
Another avenue for improvement is in the number of arithmetic operations
per oracle query. This is $\widetilde{O}(n^{2})$ for all the oracle-based
methods since each step must deal with a linear transformation. A
random process that could be faster is Coordinate Hit-and-Run. In
this, a coordinate basis is fixed, and at each step, we pick one of
the coordinate directions at random, compute the chord through the
current point in that direction and go to a random point along the
chord. It is open to show that the mixing time/conductance of this
process is polynomial and perhaps of the same order as Hit-and-Run,
thus potentially a factor of $n$ faster overall.

\paragraph{Manifold KLS.}

In \cite{lee2017convergence}, we proved the following theorem, motivated
by the goal of faster sampling and volume computation of polytopes. 
\begin{lem}
\label{lem:n-d-gaussian-manifold-isoperimetry-1}Let $\phi:K\subset\Rn\rightarrow\R$
be a convex function defined over a convex body $K$ such that $D^{4}\phi(x)[h,h,h,h]\geq0$
for all $x\in K$ and $h\in\Rn$. Given any partition $S_{1},S_{2},S_{3}$
of $K$ with $d=\min_{x\in S_{1},y\in S_{2}}d(x,y)$, i.e., the minimum
distance between $S_{1}$ and $S_{2}$ in the Riemannian metric induced
by the Hessian of $\phi$. For any $\alpha>0$, 

\[
\frac{\int_{S_{3}}e^{-\alpha\phi(x)}dx}{\min\left\{ \int_{S_{1}}e^{-\alpha\phi(x)}dx,\int_{S_{2}}e^{-\alpha\phi(x)}\,dx\right\} }\gtrsim\sqrt{\alpha}\cdot d.
\]
\end{lem}

The special case when $\phi(x)=\norm x^{2}$ and $d$ is the Euclidean
metric is given by Theorem \ref{thm:Gaussian-iso}. What are interesting
generalizations of this theorem and the KLS conjecture to the manifold
setting?

\bigskip{}
\begin{acknowledgement*}
This work was supported in part by NSF Awards CCF-1563838, CCF-1717349
and CCF-1740551. We are grateful to Ronen Eldan, Bo'az Klartag and
Emanuel Milman for helpful discussions and Yan Kit Chim for making
the illustrations. Part of this work was done while the author was
visiting the Simons Institute for the Theory of Computing. It was
partially supported by the DIMACS/Simons Collaboration on Bridging
Continuous and Discrete Optimization through NSF grant \#CCF-1740425.

\bibliographystyle{plain}
\bibliography{acg}
\end{acknowledgement*}

\end{document}